\documentclass[pdflatex,sn-mathphys-num]{sn-jnl}
\pdfoutput=1
\usepackage{pgfplots}
\pgfplotsset{compat=newest}
\usepgfplotslibrary{groupplots}
\usepgfplotslibrary{dateplot}
\usepgfplotslibrary{units}
\usetikzlibrary{spy,backgrounds}
\usepackage{pgfplotstable}
\usepackage{bm,amsmath,amssymb,nicefrac,colonequals}
\usepackage{empheq}
\usepackage{dsfont}
\usepackage{accents}
\usepackage{graphicx}
\DeclareGraphicsExtensions{.pdf,.png,.jpg}
\usepackage{color}
\usepackage{circuitikz}
\usepackage{tikz}
\usepackage{comment}
\usetikzlibrary{calc,positioning,shapes}
\usetikzlibrary{patterns,decorations.pathmorphing,decorations.markings}
\usepackage{pgfplots}
\pgfplotsset{compat=newest}
\usepackage[margin=10pt,font=small,labelfont=bf,labelsep=endash]{caption}
\usepackage{subcaption}

\usepackage{booktabs}%

\usepackage{algorithm}
\usepackage{algpseudocode}
\usepackage{longtable}
\usepackage{lscape}

\usetikzlibrary{intersections, backgrounds}
\usetikzlibrary{circuits}
\usetikzlibrary{circuits.ee.IEC}
\pgfplotsset{compat=newest}

\definecolor{color0}{rgb}{0.12156862745098,0.466666666666667,0.705882352941177}
\definecolor{color1}{rgb}{1,0.498039215686275,0.0549019607843137}
\definecolor{color2}{rgb}{0.172549019607843,0.627450980392157,0.172549019607843}
\definecolor{color3}{rgb}{0.83921568627451,0.152941176470588,0.156862745098039}
\definecolor{color4}{rgb}{0.580392156862745,0.403921568627451,0.741176470588235}
\definecolor{color5}{rgb}{0,0,0}
\definecolor{mycolor1}{rgb}{0.00000,0.44700,0.74100}
\definecolor{mycolor2}{rgb}{0.85000,0.32500,0.09800}
\definecolor{mycolor3}{rgb}{0.92900,0.69400,0.12500}
\definecolor{mycolor4}{rgb}{0.46600,0.67400,0.18800}
\definecolor{mycolor5}{rgb}{0.49400,0.18400,0.55600}

\newcommand*\closure[1]{\overline{#1}}
\newcommand{\bx}{{\bm x}}

\newcommand{\R}{\mathbb R}
\newcommand{\N}{\mathbb N}

\newcommand{\norm}[2]{\left|#1\right|_{#2}}
\newcommand{\ddt}{\partial_t}
\newcommand{\linspan}{\mathop{\rm span}\nolimits}
\newcommand{\range}{\mathop{\rm range}\nolimits}
\newcommand{\proj}{\Pi}
\newcommand{\bu}{\mathbf{u}}
\newcommand{\cor}{\calR}
\newcommand{\ubar}[1]{\underaccent{\bar}{#1}}

\newcommand{\calC}{\mathcal{C}}

\newcommand{\calJ}{\mathcal{J}}

\newcommand{\calL}{\mathcal{L}}

\newcommand{\calQ}{\mathcal{Q}}
\newcommand{\calR}{\mathcal{R}}
\newcommand{\calS}{\mathcal{S}}

\newcommand{\calU}{\mathcal{U}}

\newcommand{\calW}{\mathcal{W}}
\newcommand{\calX}{\mathcal{X}}

\newcommand{\PDE}{{\texttt{PDE}}}
\newcommand{\PDEs}{{\texttt{PDEs}}}

\newcommand{\FOM}{{\texttt{FOM}}}
\newcommand{\ROM}{{\texttt{ROM}}}
\newcommand{\ROMs}{{\texttt{ROMs}}}

\newcommand{\MOR}{{\texttt{MOR}}}
\newcommand{\OCP}{{\texttt{OCP}}}
\newcommand{\OCPs}{{\texttt{OCPs}}}

\newcommand{\POD}{{\texttt{POD}}}
\newcommand{\HAPOD}{{\texttt{HAPOD}}}
\newcommand{\BB}{{\texttt{BB}}}
\newcommand{\FE}{{\texttt{FE}}}

\newcommand{\dt}{\,\mathrm dt}

\newcommand{\Da}{\ubar\Delta}
\newcommand{\Db}{\bar\Delta}

\theoremstyle{thmstyleone}%
\newtheorem{theorem}{Theorem}
\newtheorem{lemma}[theorem]{Lemma}%
\newtheorem{corollary}[theorem]{Corollary}%

\theoremstyle{thmstyletwo}%
\newtheorem{remark}[theorem]{Remark}%

\theoremstyle{thmstylethree}%

\begin{document}

\title[Optimality-Based Control Space Reduction]{Optimality-Based Control Space Reduction for High-Dimensional Control Spaces}


\author*[1]{\fnm{Michael} \sur{Kartmann}}\email{michael.kartmann@uni-konstanz.de}

\author[1]{\fnm{Stefan} \sur{Volkwein}}\email{stefan.volkwein@uni-konstanz.de}

\affil*[1]{\orgdiv{Department of Mathematics and Statistics}, \orgname{Universität Konstanz}, \orgaddress{\street{Universitätsstraße 10}, \city{Konstanz}, \postcode{78464}, \state{State}, \country{Germany}}}

\abstract{ {We study Galerkin model reduction for unconstrained linear-quadratic optimal control
problems and show that state-space reduction alone already induces a reduced control
structure via the optimality conditions. As a result, the solely state-reduced and the combined control- and state-reduced problems are equivalent, allowing fast optimization over a reduced control space without
introducing additional approximation error. We derive lower and upper \emph{a posteriori}
error bounds for the optimal control and use them within an online-adaptive algorithm that constructs sufficiently accurate reduced spaces while solving the control problem. Convergence of the algorithm
is proven, and numerical results demonstrate that combined control and state-space
reduction yields significant speed-ups without loss of accuracy compared to state-space
reduction alone.
}}



\maketitle


\section{Introduction}\label{sec:intro}
%
In this contribution, we are interested in a certified control- and state-space Galerkin model order reduction ({\MOR}) for the solution of linear-quadratic optimal control problems (\OCPs) of the form
\begin{subequations}\label{SubEq1}
    \begin{align}\label{eq:OCPformal}
        &\min J(y,u)\coloneqq\int^{T}_{0}\tfrac{1}{2}\,\norm{C(t)y(t)-y_d(t)}{H}^2+\tfrac{\beta}{2} \norm{u(t)}{U}^2\dt
       \end{align}
        s.t. (`subject to') $(y,u)\in  L^2(0,T;V)\times L^2(0,T;U)$ solving
        \begin{align}\begin{cases}
            \ddt y(t)+A(t)y(t) = B u(t) \in V'&\text{in } (0,T),\\
            y(0)=y_0 & \text{in }H.
        \end{cases}\label{eq:FOM_PDEformal}
  \end{align}
\end{subequations}
In this introduction, we focus on formal calculations to present key ideas and refer to Sections~\ref{subsec:intro:notation} and \ref{sec:roms} for a concise setting and the definition of all involved quantities. It is well-known that the optimal control $\bar u$ satisfies the optimality condition (cf., e.g., \cite[Chapter 1]{hinze2008optimization})
\begin{align*}
    \bar u(t) = -\tfrac{1}{\beta}{\calJ_U^{-1}}B'\bar p(t)\quad\text{for }t\in  (0,T),
\end{align*}
for the adjoint variable $\bar p$ according to the cost in \eqref{eq:OCPformal}. {Here, $\calJ_U$ denotes the {Riesz} isomorphism on the control space $U$ as introduced in Section \ref{subsec:intro:notation} below}. 

{We assume that the control space $U$ and the state-space $V$ are high- or even
infinite-dimensional, making an approximation scheme for the {\OCP} necessary.
In this work, we employ Galerkin {\MOR} methods to construct a
reduced state space $V_r \subset V$, and, as is standard in \MOR, the reduced {\OCP} is obtained by projecting the
partial differential equation (\PDE) onto $V_r$. The first main contribution of this paper is the observation that a reduction of the state space alone already induces a reduced control space $U_r \subset U$, on which the optimization can be restricted without introducing any additional error. Indeed, projecting the {\PDE} in \eqref{eq:FOM_PDEformal} onto $V_r$ yields a reduced {\OCP} in which the minimizer $\bar u^r\in L^2(0,T;U)$ is still sought with values in the full control space $U$, but leads to a reduced optimal adjoint state $\bar p^r \in V_r$ satisfying the corresponding reduced optimality condition}
\begin{equation}\label{eq:red_formal}
    \bar u^r(t) = -\tfrac{1}{\beta}{\calJ_U^{-1}}B'\bar p^r(t)\in U_r\quad\text{for }t\in  (0,T).
\end{equation}
{Thus, the reduced optimal control already lies in a reduced space $U_r\subset U$ that is implicitly inherited from the reduced state space via the adjoint $\bar p^r \in V_r$ and the linear operator $-\tfrac{1}{\beta}{\calJ_U^{-1}}B'$. The optimization can therefore be restricted to the reduced control space without altering the optimization result, enabling fast optimization over a low-dimensional space while no additional error is introduced by the control-space reduction itself (see Table \ref{tab:rom_rom_comparison} in Section \ref{sec:numexps}).}

The underlying idea is closely related to Hinze’s concept of \emph{variational discretization} \cite{hinze2005variational} in the context of finite element (\FE) discretizations for control-constrained linear-quadratic {\OCPs}. In variational discretization, the control variable is not discretized explicitly. Instead, only the state and adjoint equations are discretized, while, similar to \eqref{eq:red_formal}, the control is recovered implicitly via the first-order optimality condition as a projection (due to the control constraints) of the discrete adjoint state. This approach avoids an additional control discretization error, yields optimal convergence rates for the control, and does not increase the computational complexity significantly.

Our setting differs in two essential aspects.
First, we consider model order reduction (\MOR) rather than finite element
discretization.
Second, we focus on unconstrained optimal control problems.
As a consequence, the reduced adjoint state induces an \emph{explicitly reduced}
control subspace via \eqref{eq:red_formal}, which allows the optimization to be
carried out exclusively on this reduced space.
In contrast, for control-constrained {\OCPs} the optimality condition involves
a projection onto the admissible control set, and such an explicit reduced
(linear) control space cannot be constructed.

The approach developed in this paper can therefore be interpreted as a
\emph{variational reduction of the control space}.
Beyond the practical advantages discussed above, this viewpoint also yields
theoretical benefits: the equivalence to a solely state-reduced {\OCP} improves
error estimates and simplifies convergence analyses, since no error from
an explicit control-space reduction needs to be taken into account.
A similar effect was observed in~\cite{ali2020reduced} for offline--online reduced
basis approximations of parametrized elliptic {\OCPs}. In an adaptive {\MOR} context, related conditions to \eqref{eq:red_formal} were
exploited in~\cite{kartmann2024adaptive,kartmann2025adaptive} to reduce the dimension of the parameter space, enabling the efficient construction and certification of reduced-order models (\ROMs) for inverse problems with a bilinear coupling between parameter and state.

Instead of an offline-online type model reduction of the control and state space in a parametrized setting (\cite{karcher2018certified,karcher2018reduced,bader2016certified}), we focus on an online adaptive setting for a single solve of the {\OCP} with fixed parameters. Since, for fixed parameters, an offline construction of the {\ROM} can be more costly than solving the full-order model (\FOM) {\OCP} itself, the question of how to choose the approximation spaces arises. Thus, in the absence of \emph{a priori} information, the reduced approximation spaces $V_r$ and $U_r$ must be constructed adaptively based on the current optimizer data; see, e.g., \cite{qian2017certified, keil2021non, kartmann2025adaptive}
for related approaches in \PDE-constrained parameter optimization.
{Hence, the second main contribution of this work is the development of a certified a~posteriori error–based and online adaptive algorithm that simultaneously constructs the approximation space $V_r$ (and thus also $U_r$) while solving the optimal control problem \eqref{SubEq1}. We explicitly note that the proposed adaptive algorithm can be employed to accelerate the offline
(training) phase for parametrized \OCPs, in which a large number of {\FOM}-{\OCPs} with fixed parameters
must be solved.}
To construct the reduced state space, we consider proper orthogonal decomposition (\POD) (see, e.g., \cite{GubV17}, and references therein). In this context, similar adaptive optimization strategies were considered in \cite{kone2017numerical,afanasievadaptive}, where the {\POD} model is updated based on snapshots stemming from the last suboptimal control. Another idea to construct both a {\POD} approximation for the state space and the optimal control is optimality system {\POD} (\cite{KV08,volkwein2011optimality}), where both the {\POD} problem and the optimal control problem are formulated into a combined problem, that is solved until an \emph{a posteriori} error estimator certifies a sufficient quality of the {\POD} optimizer.

\emph{A posteriori} error estimators are a widely used tool to verify the quality of the {\ROM} approximation. In the parametrized {\OCP} setting, we refer, e.g., to \cite{karcher2018certified,bader2016certified,ali2020reduced,karcher2018reduced} and to \cite{troltzsch2009pod,kartmann2024certifiedmodelpredictivecontrol} in adaptive settings. Only in \cite{ali2020reduced}, a lower \emph{a posteriori} bound for the approximation error of the whole optimal triple was given, which can be used to verify the sharpness of the upper bound. Further, the estimation technique used in \cite{troltzsch2009pod} does not use the optimality condition of the reduced model and thus can be applied to any arbitrary control in the control space. This is important since, usually, the solution of {\OCPs} is found by iterative methods that satisfy the optimality condition up to a tolerance, which might conflict with the optimality assumptions needed to make the \emph{a posteriori} estimate rigorous. We present both lower and upper \emph{a posteriori} error bounds that can be used to quantify the distance between the {\FOM} optimal control and an arbitrary element in the control space and thus overcome the aforementioned conflict for {\OCPs} without (control) constraints. {We highlight that, similar to \cite{troltzsch2009pod}, the evaluation of the presented bounds depends on {\FOM} quantities, but not on the solution of the {\FOM}-{\OCP} itself. Hence, the bounds do not achieve online efficiency, which is, however, not needed, in the online-adaptive setting we consider. Indeed, the {\FOM} data used to evaluate the bounds can be reused to update the \ROM, if necessary (see Remark \ref{rem:apost}).}

Let us mention that other techniques for control (or action) space reduction can be found in, e.g., \cite{delavari2025action,chun2024multiscale,8014482} and the references therein.

Next, we summarize our contribution here:
\begin{itemize}
    \item We show that a Galerkin reduction for the state space directly leads to a reduced structure of the optimal control. We prove that the control- and state-reduced {\OCP} is equivalent to the state-reduced {\OCP} (cf. Lemma~\ref{lem:equival}) and convergence towards the solution of the {\FOM-\OCP} as the reduced dimension tends to infinity (cf. Lemma~\ref{lem:ROMconv}).
    \item We provide both lower and upper \emph{a posteriori} bounds for the error in the optimal control with respect to an arbitrary control (cf. Corollary~\ref{cor:error_est_control}).
    Further, we provide an error representation for the error in the optimal value function, in which there is no contribution stemming from the reduction of the control space (cf. Theorem~\ref{theo:optimal_val}).
    \item Based on these error estimates, we introduce an adaptive algorithm (Algorithm~\ref{alg:ROM_OPT}) to approximate the {\FOM} optimal control using {\POD} model reduction for the state and control space. Using interpolation and convergence properties of the {\ROM}, we prove convergence of this algorithm in separable Hilbert spaces (Theorem~\ref{theo:convergence_algo}).
    \item We verify our theoretical findings numerically and demonstrate the error-free speed-up stemming from a combined control- and state-space reduction in contrast to solely state-space reduction. The code for the numerical experiments is provided in \cite{code}.
\end{itemize}
The paper is organized as follows.
After introducing some basic notation in Section~\ref{subsec:intro:notation}, 
we present the {\FOM} and the {\ROMs} in Section~\ref{sec:roms}. 
The results on \emph{a~posteriori} error estimation are formulated in Section~\ref{sec:apost}. 
Section~\ref{sec:adaptivealgo} discusses the adaptive algorithm and its convergence, 
while its numerical realization is presented in Section~\ref{sec:numexps}.

\subsection{Setting and notation}\label{subsec:intro:notation}

Throughout the paper, let $U,V,H$ be real, separable Hilbert spaces with $V\subset H$ compactly and densely embedded. The space $\calL(X,Y)$ is the Hilbert space of all linear bounded operators between the Hilbert spaces $X$ and $Y$, and $\calJ_X:X\to X'$ is the Riesz isomorphism. The identity on a Hilbert space $X$ is denoted by $I$. For ${S}\in \calL(X,Y)$ we denote its induced norm by $\norm{{S}}{}\coloneqq \sup\{|{S}x|_Y:|x|_X=1\}$. For $T>0$ the state trajectory space is denoted by $W(0,T)\coloneqq H^1(0,T;V')\cap L^2(0,T;V)$, where $V'$ denotes the topological dual space of $V$. The norm and the scalar product are denoted by $|x|_X$ and $\langle x,\tilde x\rangle_{X}$, respectively, for $x,\tilde x\in X$. The dual pairing is denoted by $\langle x', x\rangle_{X',X}$ for $x'\in X'$. Given a closed set $K\subset X$, we define the orthogonal projection $\Pi_{K}^X:X\to K$, $v\mapsto \Pi_{K}^Xv$, as the unique solution to
\begin{align*}
    {\langle v-\Pi_{K}^Xv, w\rangle}_X=0\quad\text{for all }w\in K.     
\end{align*}
For an operator $B\in \calL(U,V')$ we define the dual operator $B'\in \calL(V,U')$ by
\begin{align*}
{\langle B'v, u\rangle}_{U',U}={\langle Bu, v\rangle}_{V',V}\quad\text{for all }u\in U\text{ and }v\in V. 
\end{align*}
Reduced quantities are marked by the symbol $r$, e.g., $U_r\subset U$, $V_r\subset V$ denote reduced spaces and $y^r\in W^r(0,T)$ a reduced state in the reduced state trajectory space $W^r(0,T)\coloneqq H^1(0,T;V_r')\cap L^2(0,T;V_r)$. {For an overview of the notation used, we refer to Table \ref{tab:notation}.}
\begin{table}[ht!] 
    \scriptsize
	\centering 
    \caption{{Notation table}}
	\label{tab:notation}
	\begin{tabular}{ll}\toprule
        Symbol & Definition and reference \\
        \midrule
        $\calL(X,Y)$  & space of bounded, linear operators (Section \ref{subsec:intro:notation}) \\
        $H,V$ & state spaces with $V\subset H$ (Section \ref{subsec:intro:notation}) \\
        $V_r$ & reduced state space (Section \ref{subsec:rom})\\ 
        $W(0,T)$  &  state trajectory space (Section \ref{subsec:intro:notation}) \\
        $U$ & control space (Section \ref{subsec:intro:notation}) \\
        $U_r$ & reduced control space \eqref{eq:ONB} \\
        $L^2(0,T;U)$ & control trajectory space (Section \ref{subsec:intro:notation})\\
         $\calJ_U$  & {Riesz} isomorphism $\calJ_U:U\to U' $ (Section \ref{subsec:intro:notation})\\
         $y,p,u$ & state, adjoint state, control (Section \ref{subsec:foms}) \\
         $y^r,p^r,u^r$ & reduced state, adjoint state, control (Section \ref{subsec:rom}) \\
         $J$  & cost function in $y,u$ \eqref{eq:OCP} \\
         $\hat J$  & state-eliminated cost function in $u$ \eqref{redOCP} \\
          $\hat J^r$  & reduced state-eliminated cost function in $u$ \eqref{eq:ROM_OCP}\\
          $\min\limits_{u\in L^2(0,T;U)} \hat J(u)$  & {\FOM-\OCP} \eqref{redOCP}\\
           $\min\limits_{u\in L^2(0,T;U)} \hat J^r(u)$  & state-reduced {\OCP} \eqref{SubEq3} \\
            $\min\limits_{u\in L^2(0,T;U_r)} \hat J^r(u)$ & state- and control-reduced {\OCP} \eqref{eq:fullROM_OCP}   \\
        $A,B,C$  & system operators (Section \ref{subsec:foms})\\
         $L$ & Lagrange function \eqref{eq:lagrange}\\
		\bottomrule
	\end{tabular}
\end{table}
%
\section{Reduced-order models (\ROMs) with control- and state-space reduction}\label{sec:roms}

In this section, we introduce the {\FOM} and the control and state {\ROM}.

\subsection{Full-order model (\FOM)}\label{subsec:foms}

We consider time-varying optimal control problems of the form
\begin{subequations}\label{SubEq2}
    \begin{align}\label{eq:OCP}
        &\min J(y,u)\coloneqq\int^{T}_{0}\tfrac{1}{2}\,\norm{C(t)y(t)-y_d(t)}{H}^2+\tfrac{\beta}{2} \norm{u(t)}{U}^2\dt
        \end{align}
        s.t. $(y,u)\in W(0,T)\times L^2(0,T;U)$ solving
        \begin{align}\begin{cases}
            \ddt y(t)+A(t)y(t) = B u(t) \in V'&\text{in } (0,T),\\
            y(0)=y_0 & \text{in }H.
        \end{cases}\label{eq:FOM_PDE}
    \end{align}
\end{subequations}
We make the following assumptions throughout this paper:
\begin{enumerate}
    \item [] \textbf{(A1)}: $A\in L^\infty(0,T;\calL(V,V'))$ to be coercive uniformly in $t$, in the sense that there exists a constant $\eta_V>0$ with
    \begin{equation}\label{eq:coercivity}
        {\langle A(t)v,v\rangle}_{V',V}\geq \eta_V\norm{v}{V}^2 \quad \text{for all }t\in (0,T)\text{ and }v\in V;
    \end{equation}
    \item [] \textbf{(A2)}: $C\in L^\infty(0,T;\calL(H,H))$, $B\in \calL(U,V')$, $\beta>0$, $y_d\in L^2(0,T;H)$, and $y_0\in H$.
\end{enumerate}
\begin{remark}
{
    The measurability in the assumption $A\in L^\infty(0,T;\calL(V,V'))$ (and analogously $C\in L^\infty(0,T;\calL(H,H))$) is understood in the sense that the scalar function $t\mapsto \langle A(t)v,w\rangle_{V',V} $ is measurable for all $v,w\in V$.
     The measurability assumption is needed to make the integral expressions well-defined, e.g., for the weak solution concept $y\in W(0,T)$ used throughout the paper, and the definition of the Lagrange functional in \eqref{eq:lagrange}.}
\end{remark}
Note that these assumptions are made for simplicity and can be relaxed to, e.g., more general observation operators and unstable evolution operators (see Remark \ref{rem:1} below).
Under these assumptions, \eqref{eq:FOM_PDE} and \eqref{SubEq2} are well-posed. To be precise, for every control $u\in L^2(0,T;U)$ there exists a unique state $y=y(u)\in W(0,T)$ (cf., e.g., \cite[Chapter 7]{Eva10}). Therefore, the cost function
\begin{align*}
    \hat J(u)\coloneqq J(y(u),u)\quad\text{for }u\in L^2(0,T;U)
\end{align*}
is well-defined and \eqref{SubEq2} is equivalent to the problem
\begin{align}
    \label{redOCP}
    \min \hat J(u)\quad\text{s.t.}\quad u\in L^2(0,T;U).
\end{align}
We call \eqref{redOCP} the {\FOM-\OCP}. Moreover, the reduced problem is strictly convex and admits a unique minimizer $\bar u\in L^2(0,T;U)$ (cf., e.g., \cite[Chapter 1]{hinze2008optimization}). To formulate the corresponding optimality condition, we introduce the \emph{adjoint equation} for $y\in L^2(0,T;H)$ as
\begin{equation}
\label{eq:FOMadjoint}
\begin{cases}
-\ddt {p}(t)+A'(t) p(t) = C'(t)(C(t)) y(t)-y_d(t))\in V'  &\text{in } (0,T),\\
 p(T)= 0 & \text{in }H.
\end{cases}
\end{equation}
We infer from \textbf{(A1)} and \textbf{(A2)} that there exists a unique solution $p=p(y)\in W(0,T)$ for any $y\in L^2(0,T;V)$. Therefore, the minimizer
$\bar u $ is characterized by the first-order sufficient optimality condition
\begin{equation}\label{eq:FOM_optcond}
    \nabla \hat J(\bar u) = \calJ_{U}^{-1}B'\bar p+\beta \bar u = 0,
\end{equation}
where the optimal adjoint state $\bar p=p(\bar y)\in W(0,T)$ is given as the solution of \eqref{eq:FOMadjoint} for the optimal state $\bar y=y(\bar u)$, solving \eqref{eq:FOM_PDE} for $u = \bar u$. {Note that in \eqref{eq:FOM_optcond}, $\calJ_U^{-1}$ denotes the {Riesz} isomorphism on the control space $U$ as introduced in Section \ref{subsec:intro:notation}.}
\begin{remark}\label{rem:1}
    For the construction of the reduced control space $U_r$ below, the operator $B$ in \eqref{eq:FOM_PDE} is chosen to be time-independent. However, as can be seen in \eqref{eq:keyequality} below, it can have an affine time-dependent structure, e.g., $B(t)=b(t)B$ for some $b\in L^\infty(0,T;\R)$.
    In this case, \eqref{eq:coercivity} can be replaced by the Garding inequality without loss of generality (w.l.o.g.), that is, with the existence of $\eta_V>0$ and $\eta_H\geq 0$, such that
    \begin{equation}\label{eq:coercivity_gen}\nonumber
        {\langle A(t)v,v\rangle}_{V',V}\geq \eta_V\norm{v}{V}^2-\eta_H \norm{v}{H}^2\quad \text{for all }t\in (0,T)\text{ and }v\in V.
    \end{equation}
    Further, the framework developed below directly applies to more general observation operators $C\in L^\infty(0,T;\calL(H,\tilde H))$ for some output Hilbert space $\tilde H$. Incorporating maximal parabolic regularity theory and interpolation-space embeddings can enable the use of more general observation operators $C$.
\end{remark}
\subsection{Reduced-order model (\ROM)}\label{subsec:rom}

The difficulty in an efficient numerical treatment of \eqref{eq:OCP} arises from the infinite dimensionality (or high-dimensionality after discretization) of the control space $U$ and the state space $V$.

A standard {\MOR} approach to accelerate computations is to project the {\PDE} \eqref{eq:FOM_PDE} onto a reduced (low-dimensional) subspace $V_r=\linspan(v_1,\ldots,v_r)\subset V$ for $r\in \N$. This leads to the \emph{state-reduced} \OCP
\begin{subequations}\label{SubEq3}
    \begin{align}\label{eq:ROM_OCP}
        \min \hat J^r(u)\coloneqq \int_0^T\tfrac{1}{2}\norm{C(t)y^r(t)-y_d(t)}{H}^2+\tfrac{\beta}{2}\norm{u(t)}{U}^2\dt\quad\text{s.t.}\quad u\in L^2(0,T;U),
    \end{align}
    where $y^r=y^r(u)\in W^r(0,T)$ solves the reduced state equation
    \begin{align}
        \begin{cases}
            \ddt y^r(t)+A(t)y^r(t) = Bu(t)\in V_r'  &\text{in } (0,T),\\
            y^r(0)=y_0^r & \text{in }H.
        \end{cases}\label{eq:ROM_PDE}
    \end{align}
\end{subequations}
Here, $y_0^r=\Pi_{V_r}^H y_0\in V_r$ is the orthogonal projection onto $V_r$ w.r.t. the inner product of $H$. The well-posedness of the reduced state equation and the state-reduced {\OCP} are inherited from the well-posedness of the \FOM. The solution to \eqref{eq:ROM_PDE} can be represented as $y^r=\sum_{i=1}^r \mathbf{y}^r_i v_i$ with $\mathbf{y}\in H^1(0,T;\R^r)$ (cf. \cite[Proposition 1.27]{GubV17}). As before, the optimality condition for the reduced unique minimizer $\hat u^r\in L^2(0,T;U)$ reads
\begin{equation} \label{eq:halfrom_opcond}
    \nabla\hat J^r(\hat u^r) = \calJ_{U}^{-1}B'\hat p^r+\beta \hat u^r = 0,
\end{equation}
for $\hat p^r=\sum_{i=1}^r \hat{\mathbf{p}}^r_i v_i$ with $\hat{\mathbf{p}}^r\in  H^1(0,T;\R^r)$ solving
\begin{equation}
\label{eq:ROMadjoint}
\begin{cases}
-\ddt \hat p^r (t)+A'(t) \hat p^r(t) = C'(t)(C(t)\hat y^r(t)-y_d(t))\in V_r'  &\text{in } (0,T),\\
\hat p^r(T)= 0 & \text{in }H,
\end{cases}
\end{equation}
for $\hat y^r$ solving \eqref{eq:ROM_PDE} for $u=\hat u^r$. The optimality condition \eqref{eq:halfrom_opcond} implies
\begin{equation}\label{eq:keyequality}
    \hat u^r=-\frac{1}{\beta}\calJ_{U}^{-1}B'\hat p^r=-\frac{1}{\beta}\sum_{i=1}^r \hat{\mathbf p}^r_i\calJ_{U}^{-1}B'v_i.
\end{equation}
By setting
\begin{subequations}
    \begin{align}
        u_i & =  \calJ_{U}^{-1}B'v_i&&\hspace{-30mm}\text{for }i=1,\ldots,r\label{eq:Ur_construction},\\
        \hat{\mathbf u}^r_i&=-\tfrac{1}{\beta}\hat{\mathbf p}^r_i&&\hspace{-30mm}\text{for }i=1,\ldots,r\label{eq:Ur_construction_coeff},
    \end{align} 
\end{subequations}
we can write the optimality condition \eqref{eq:halfrom_opcond} as 
\begin{equation}\label{eq:keyequality2}
    \hat u^r=\sum_{i=1}^r \hat{\mathbf u}^r_i u_i.
\end{equation}
This means that the optimal control of the state-reduced {\OCP} already lies in a reduced control space $\hat u^r\in L^2(0,T;U_r)$ with $U_r=\linspan(u_1,\ldots,u_r)$. {Note that in general, it holds $\dim U_r \leq r$ , and if the operator $B\in \calL(U,V_r')$ is not surjective, and thus $B'\in \calL(V_r, U')$ not injective, the $u_i$'s are linearly dependent and one has $\dim U_r < \dim V_r$. Applying, e.g., Gram-Schmidt orthogonalization leads to a basis $(\tilde u_i)_{i=1}^{r_u}$ with 
\begin{equation}\label{eq:ONB}
    U_r=\linspan(\tilde u_1,\ldots,\tilde u_{ r_u}) \quad \text{for } r_u \coloneqq \dim U_r \leq r.
\end{equation}}
Due to \eqref{eq:keyequality2} we consider the \emph{control- and state-reduced} \OCP
\begin{align}\label{eq:fullROM_OCP}
    \min \hat J^r(u^r)\quad\text{s.t.}\quad {u^r}\in L^2(0,T;U_r),
\end{align}
where $y^r=y^r(u)$ solves the reduced state equation \eqref{eq:ROM_PDE}.
Since $U_r\subset U$ is a Hilbert space itself, \eqref{eq:fullROM_OCP} admits the unique minimizer $\bar u^r$ by assumptions \textbf{(A1)} and \textbf{(A2)}. Since the state-reduced minimizer already satisfies $\hat u^r\in L^2(0,T;U_r)$ by \eqref{eq:keyequality2}, we obtain $\hat u^r=\bar u^r$ by uniqueness of both minimizers.

\begin{lemma}\label{lem:equival}
    We have $\bar u^r = \hat u^r$. In particular, the associated optimal states and optimal adjoints are equal, that is, $\bar y^r\coloneqq\hat y^r$ and $\bar p^r \coloneqq \hat p^r$.
\end{lemma}

Hence, in a numerical implementation, one can always restrict oneself to the control and state-reduced {\OCP} \eqref{eq:fullROM_OCP}.
\begin{remark}[Optimality condition of the control- and state-reduced OCP]\label{eq:numerical_solution_state_reduced_OCP}
    Given the reduced control space $U_r$ as in \eqref{eq:ONB}, we define the positive definite product matrix $\mathbf{M}_U\in \R^{{r_u\times r_u}}$ via
    \begin{align*}
        \big(\mathbf{M}_U\big)_{ij} \coloneqq{\langle \tilde u_j,\tilde u_i\rangle}_{U}\quad\text{for } i,j=1,\ldots,{r_u},
    \end{align*}
    and the discrete control space $\R_U^{{r_u}}$, that is $\R^{{r_u}}$ endowed with the product
    \begin{align*}
        {\langle \bu, \mathbf v\rangle}_{\R_U^{{r_u}}}\coloneqq \bu^\top \mathbf{M}_U \mathbf v\qquad\text{for }\bu, \mathbf v\in \R^{{r_u}}.
    \end{align*}
    To write the optimality condition of \eqref{eq:fullROM_OCP} in an unconstrained manner, we eliminate the constraint $u(t)\in U_r$, by defining the operator
    \begin{align*}
        B_r: \R_U^{{r_u}}\to V',\quad\bu \to B\bigg( \sum_{i=1}^{{r_u}} \tilde u_i \bu_i\bigg).
    \end{align*}
    Then, \eqref{eq:fullROM_OCP} is equivalent to
    \begin{subequations}\label{SubEq5}
        \begin{align}\label{eq:fullROM_OCP_dis}
            &\min_{}\mathbf{\hat J}^r({\bu})\coloneqq \int_0^T\tfrac{1}{2}\norm{C(t)y^r(t)-y_d(t)}{H}^2+\tfrac{\beta}{2}\norm{\bu(t)}{\R^{{r_u}}_U}^2\mathrm dt~\text{ s.t. }\bu\in L^2(0,T;\R_{U}^{{r_u}}),
        \end{align}
        where $y^r=y^r(\bu)$ solves 
        \begin{align}
            \begin{cases}
                \ddt y^r(t)+A(t)y^r(t) = B_r\bu(t)\in V_r'  &\text{in } (0,T),\\
                y^r(0)=y_0^r & \text{in }H 
            \end{cases}\label{eq:fullROM_PDE_dis} 
        \end{align}
   \end{subequations}
    with the following sufficient optimality condition for the unique minimizer $\bar \bu^r\in \R^{{r_u}}$
    \begin{align*}
        \nabla \mathbf{\hat J}^r(\bar\bu^r)=\bar \bu^r + \mathbf{M}_U^{-1}B_r'\hat p^r=0
    \end{align*}
    with $\hat p^r$ as in \eqref{eq:ROMadjoint} and $\bar \bu^r = \hat \bu^r$ for $\hat \bu^r$ as in \eqref{eq:Ur_construction_coeff}.
\end{remark}

\section{A posteriori error estimation}\label{sec:apost}

While {\OCP} \eqref{eq:fullROM_OCP} can be advantageous for efficient numerical realization, the equivalent formulation in {\OCP} \eqref{SubEq3} is primarily useful for theoretical purposes: error analysis and convergence results for the control- and state-reduced problem can be traced back to the solely state-reduced case. 

\subsection{Error estimator for the optimal control}

First, we recall the well-known residual-based \emph{a posteriori} error estimator for linear coercive operators. {A proof was given, e.g., in \cite[Proposition 2.19, Proposition 2.21]{haasdonk2017reduced}, but, for the sake of completeness, a proof is presented in Appendix \ref{ap:prooflemma4}.}

\begin{lemma}\label{lem:lemApost}
    Let $\calU$ be a real Hilbert space, $d\in \calU'$ and $\calQ\in \calL(\calU,\calU')$ coercive with constant $\beta>0$, that is
    \begin{align*}
        {\langle\calQ u,u\rangle}_{\calU',\calU}\ge\beta\,{|u|}_\calU^2\quad\text{for all }u\in\calU.
    \end{align*}
    Let $u\in \calU$ satisfy the variational problem
    \begin{equation}
        \label{eq:coercive_est_standard}
        \calQ u=d
    \end{equation}
    and $\tilde u\in \calU$ be given. Then we have the error-residual equivalence 
    \begin{equation}
        \label{eq:relation_coercive_est_standard}
        \tfrac{1}{\norm{\calQ}{}}\norm{d-\calQ\tilde u}{\calU'}\leq \norm{u-\tilde u}{\calU} \leq \tfrac{1}{\beta}\norm{d-\calQ\tilde u}{\calU'}.
    \end{equation}
\end{lemma}
Next, we cast the {\FOM} optimality condition \eqref{eq:FOM_optcond} into the form \eqref{eq:coercive_est_standard} for $\calU=L^2(0,T;U)$. Therefore, let $\hat y$ be the solution to the {\FOM 
-\PDE} \eqref{eq:FOM_PDE} for $u=0$. Then we can write each solution of \eqref{eq:FOM_PDE} as $y(u)=\calS u+\hat y$, where $\calS\in \calL(L^2(0,T;U), W(0,T))$ is the linear solution operator for \eqref{eq:FOM_PDE} for $y_0=0$. Defining $\hat y_d \coloneqq y_d-C\hat y$, we see that \eqref{redOCP} can be written as
 \begin{align}
     \min\hat J(u)=\tfrac{1}{2}\norm{C\calS u-\hat y_d}{L^2(0,T;H)}^2+\tfrac{\beta}{2}\norm{u}{L^2(0,T;U)}^2\text{ s.t. }u\in L^2(0,T;U),
 \end{align}
and hence \eqref{eq:FOM_optcond} as
\begin{align*}
    \nabla\hat J(\bar u)=\beta \bar u + \calJ_U^{-1}\calS'C'(C\calS\bar u-\hat y_d) =0.
\end{align*}
Defining $\calQ: \calU \to \calU'\simeq L^2(0,T;U')$ and $d\in \calU'$ as
\begin{align}\label{eq:Q}
    \calQ &\coloneqq \beta I+ \calS'C'C\calS,\\
    d &\coloneqq \calS'C'\hat y_d,\label{eq:d}
\end{align}
we obtain the optimality condition in the form \eqref{eq:coercive_est_standard} for $u= \bar u$.
\begin{corollary}[Error estimator for the first-order optimality condition]\label{cor:error_est_control}
    Let $\bar u\in L^2(0,T;U)$ be the solution of \eqref{SubEq2} and $\tilde u\in L^2(0,T;U)$ be arbitrary. Defining the upper and lower bounds, and true error
    \begin{align*}
        \Db(\tilde u)\coloneqq & \tfrac{1}{\beta}\norm{\nabla\hat J(\tilde u)}{L^2(0,T;U)},\quad\Da(\tilde u)\coloneqq \tfrac{1}{\beta+\norm{\calC\calS}{}^2}\norm{\nabla\hat J(\tilde u)}{L^2(0,T;U)},\\
        e_u(\tilde u)\coloneqq & \norm{\bar u -\tilde u}{L^2(0,T;U)},
    \end{align*}
    we have the \emph{a posteriori estimate}
    \begin{equation}\label{eq:error_est}
         \Da(\tilde u)\leq e_u(\tilde u) \leq  \Db(\tilde u).
    \end{equation}
\end{corollary}

\begin{proof}
    Note that $\calQ\in \calL(L^2(0,T;U),L^2(0,T;U)')$ -- defined in \eqref{eq:Q} -- is coercive, since $\beta >0$ by assumption \textbf{(A2)} and $\calS'C'C\calS$ is non-negative. Now we apply Lemma~\ref{lem:lemApost} to $\calU = L^2(0,T;U)$, $\calQ$, $b$ as in \eqref{eq:Q}, \eqref{eq:d}, and $\tilde u \in \calU$. Evaluating $\calJ_U^{-1}(\calQ\bar u^r-d)=\nabla\hat J(\tilde u)$ and estimating the norm of $\calQ$, implies the result.
\end{proof}

\begin{remark}\label{rem:apost}
    Note that the upper bound in \eqref{eq:error_est} is the same as the perturbation bound given in \cite{troltzsch2009pod} in the case without control constraints.
    We are particularly interested in the case $\tilde u = \bar u^r$. For this special choice of $\tilde u$, the proof given in, e.g., \cite[Theorem 3.7]{kartmann2024certifiedmodelpredictivecontrol}, is also possible since $\bar u^r$ fulfills the optimality condition of the state-reduced \OCP. However, in contrast to \cite{kartmann2024certifiedmodelpredictivecontrol,karcher2018certified, ali2020reduced}, the proof given here completely avoids using the optimality condition of the reduced model ($\tilde u$ can be arbitrary), and directly provides a lower bound without additional effort. This is a benefit in practice, since we can apply the error estimators in \eqref{eq:error_est} to each (possibly non-optimal) iterate $u_k$ of an iterative algorithm ($\tilde u=u_k$) and do not have to solve the {\ROM} up to optimality. {We explicitly note that the evaluation of the error estimator in \eqref{eq:error_est} requires a calculation of the {\FOM} state and adjoint at $\tilde u$ and thus, is not online-efficient. However, in the adaptive setting introduced below in Algorithm~\ref{alg:ROM_OPT}, they can serve as new data for updating the reduced-order model. 
    If one is interested in obtaining an offline-online decomposable upper bound of $\bar \Delta(\tilde u)$, one can add $\pm \nabla J^r(\tilde u)$ on the right-hand side of \eqref{eq:error_est} to estimate
    \begin{equation}\nonumber
        \Db(\tilde u)\leq \tfrac{1}{\beta}\norm{\nabla\hat J^r(\tilde u)}{}+\tfrac{1}{\beta}\norm{\nabla\hat J^r(\tilde u)-\nabla\hat J(\tilde u)}{}=\tfrac{1}{\beta}\norm{\nabla\hat J^r(\tilde u)}{}+\tfrac{1}{\beta}\norm{B'(p^r(y^r(\tilde u))-p(y(\tilde u))}{}
    \end{equation}
    The first term can be computed cheaply, while the second can be estimated via residual-based a posteriori estimators for state and adjoint (cf., e.g., \cite{karcher2018certified,kartmann2024certifiedmodelpredictivecontrol}). Using offline-online decomposition, their evaluation also depends solely on the \ROM. As far as we can see, such a construction is not available for the lower bound.}
    Finally, we note that incorporating maximal parabolic regularity theory and interpolation-space embeddings can lead to sharper estimates of $\norm{CS}{}$ in specific settings, which would improve the lower bound.
\end{remark}
%

\subsection{Optimal value function error representation}

In this section, we provide an error representation for the error in the optimal value function based on \cite{ran}, in which the error terms, which depend on the control-space reduction, vanish (see \eqref{eq:val_error_repr}).

We introduce the space $\calX\coloneqq W(0,T) \times L^2(0,T;U) \times W(0,T)$ endowed with the standard product topology. In addition, the Lagrange function $L$ is given by
\begin{align}\label{eq:lagrange}
    {L(x)= J(y,u)+{\langle Bu-Ay-\dot y,p\rangle}_{L^2(0,T;V'),L^2(0,T;V)}
     + \langle y_0-y(0), p(0)\rangle_{H}}
\end{align}
 for $x=(y,u,p)\in \calX$. Let $\bar u\in U$ be the optimal solution to \eqref{SubEq2}, $\bar y=y(\bar u)$ the associated optimal state (that is, $\bar y$ solves \eqref{eq:FOM_PDE} for $u=\bar u$) and $\bar p=p(\bar y)$ the associated optimal adjoint (that is, $\bar p$ solves \eqref{eq:FOMadjoint} for $y=\bar y$). We set $\bar x=(\bar y,\bar u, \bar p)\in \calX$. The {\FOM} optimality system  \eqref{eq:FOM_PDE} for $u=\bar u$, \eqref{eq:FOM_optcond}, and \eqref{eq:FOMadjoint} for $y=\bar y$, can be compactly written as
\begin{equation}\label{eq:FOM_opt_cond}
    L'(\bar x)=0 \quad\text{in } \calX'.
\end{equation}
Similarly, let $\hat u^r\in$ be the optimal solution to \eqref{SubEq3} and $\hat y^r=y^r(\hat u^r)$ and $\hat p^r=p^r(\hat y^r)$ the associated optimal state and adjoint, respectively. Then we set $\hat  x^r=(\bar y^r, \hat u^r, \bar p^r)\in \hat\calX$ with $\hat\calX\coloneqq W^r(0,T) \times L^2(0,T;U)\times W^r(0,T)$ supplied with the standard product topology. The state-reduced {\ROM} optimality system, \eqref{eq:ROM_PDE} for $u=\hat u^r$, \eqref{eq:halfrom_opcond}, and \eqref{eq:ROMadjoint}, is given by
\begin{align}\label{eq:ROM_opt_cond}
    L'(\hat x^r)=0 \quad\text{in } \hat\calX'.
\end{align}
Note that due to Lemma~\ref{lem:equival}, we have $\hat x^r=\bar x^r\coloneqq (\bar y^r,\bar u^r, \bar p^r)\in \calX_r\coloneqq W^r(0,T) \times L^2(0,T;U_r)\times W^r(0,T)$. Note that $\hat x^r$ solves the optimality system \eqref{eq:ROM_opt_cond} in variational form on $\hat\calX'$ without approximation of the control space $L^2(0,T;U)$. Hence, the {\FOM} optimal control $\bar u$ can be used as a test function in \eqref{eq:argument} below to make the control approximation term vanish in \eqref{eq:val_error_repr}.
\begin{theorem}[Optimal value function error representation]
    \label{theo:optimal_val}
    For $\bar u\in L^2(0,T;U)$ and $\bar u^r\in L^2(0,T;U_r)$, we have the following error representation 
    \begin{equation}\label{eq:val_error_repr}
        \hat J(\bar u)-\hat J^r(\bar u^r)= \tfrac{1}{2}\inf\big\{ L'_y(\bar x^r)(\bar y-y^r)+ L'_p(\bar x^r)(\bar p-p^r)\,\big|\,y^r,p^r\in W^r(0,T)\big\},
    \end{equation}
    where $L_z'(\bar x^r)$ is the partial Fr\'echet derivative of $L$ at $\bar x^r$ for $z\in \{y,u,p\}$.
\end{theorem}

\begin{proof}
    Let $e=(e_y,e_u,e_p)\coloneqq \bar x-\hat x^r\in\calX$. By the definition of $L$, $\bar x$, and $\hat x^r$, Lemma~\ref{lem:equival}, and the fundamental theorem of calculus, we have
    \begin{align*}
        \mathsf{err}_{\hat J}\coloneqq \hat J(\bar u)-\hat J(\bar u^r)=\hat J(\bar u)-\hat J(\hat u^r)=L(\bar x)-L(\hat x^r)=\int_{0}^1 L'(\hat x^r+t e)( e)\dt.
    \end{align*}
    Adding and subtracting $\nicefrac{L'(\hat x^r)(e)}{2}$ and using \eqref{eq:FOM_opt_cond} tested with $e\in \calX$, leads to
    \begin{align*}
         \mathsf{err}_{\hat J} =& \int_{0}^1L'(\hat x^r+t e)(e)\dt \pm \frac{1}{2} L'(\hat x^r)(e) -\frac{1}{2}L'(\bar x)(e)
         =\frac{1}{2}L'(\hat x^r)(e), 
    \end{align*}
    since the integral and the terms $-\nicefrac{L'(\hat x^r)(e)}{2}$ and $-\nicefrac{L'(\bar x)(e)}{2} $ cancel, the latter are a trapezoidal approximation of the former and $L'$ is linear (cf. \cite{ran}). Due to the {\ROM} optimality condition \eqref{eq:ROM_opt_cond} for $\hat x^r$, the last term can be formulated for all $\tilde x=(\tilde y, \tilde u, \tilde p)\in \hat \calX$ as
    \begin{align}\label{eq:argument}
        L'(\hat x^r)(e)= L'(\hat x^r)(\bar x-\tilde x)=L_y'(\hat x^r)(\bar y-\tilde y)+L_p'(\hat x^r)(\bar p-\tilde p)+L_u'(\hat x^r)(\bar u-\tilde u).
    \end{align}
    Since $\hat x^r$ solves the solely state-reduced optimality condition, we can choose $\tilde x=(y^r,\bar u, p^r)\in \hat \calX$ for arbitrary $y^r,p^r\in W^r(0,T)$. Hence, the term depending on $L_u'(\hat x^r)$ vanishes and since $y^r,p^r\in W^r(0,T)$ are arbitrary, we obtain \eqref{eq:val_error_repr}.
\end{proof}
Similar to, e.g., \cite{azmi2025stabilizationparabolictimevaryingpdes,karcher2018certified}, one can obtain a residual-based a posteriori error estimator from Theorem~\ref{theo:optimal_val}, which can be evaluated in practice. 
%
\section{Combined adaptive state and control reduction}\label{sec:adaptivealgo}
%
As we have seen in Section~\ref{subsec:rom}, the space $U_r$ is determined by the space $V_r$. Unfortunately, good spaces $V_r$ are, in general, not known a priori. Moreover, an offline construction of $V_r$ usually depends on full-order model solves and can be too costly for a single solve of {\FOM-\OCP}, even though it might pay off in parametrized settings. Hence, for a single solve of {\FOM-\OCP} (with fixed parameters), we propose the following adaptive Algorithm~\ref{alg:ROM_OPT}, that only needs an initial guess $u_0\in L^2(0,T;U)$ to simultaneously construct an approximate solution for the {\FOM-\OCP} as well as corresponding approximation spaces $U_r$ and $V_r$. Note that this algorithm can also be used to speed up the offline training for {\ROMs} for parametrized {\OCPs} to avoid solving {\FOM-\OCPs} for many parameters, or, if necessary, to refine an initially given space $V_r$.
%
\subsection{Construction of the reduced spaces with \POD}\label{subsec:POD}

Let $r\in \N$. To construct the space $V_r={\linspan(v_1,\ldots, v_r)}$ out of a countable data (or \emph{snapshot}) set $S\subset L^2(0,T;V)$, we focus on \POD
\begin{equation}\label{eq:POD}
    \min\sum_{s\in S}\big|s-\proj_{V_r}^Vs\big|_{L^2(0,T;V)}^2\quad \text{s.t.}\quad\{v_i\}_{i=1}^r\subset V,~{\langle v_i,v_j\rangle}_V=\delta_{ij}~(i,j=1,\ldots, r).
\end{equation}
The solution to this problem can be found using the linear correlation operator $\cor_S:V\to V$
\begin{equation}\label{eq:cor_op}
    \cor_Sv \coloneqq \sum_{s\in S}\int_{0}^T{\langle s(t),v\rangle}_V\,s(t) \dt \quad \text{for }v\in V.
\end{equation}
We have the following result.

\begin{theorem}\label{theo:POD}
    The operator $\cor_S\in \calL(V,V)$ is non-negative, self-adjoint, and compact.
    Define the {\POD} space 
    \begin{equation}\label{eq:closureW}
        \calW_S \coloneqq \closure{\range(\cor_S)}\subset V.
    \end{equation}
    Let $\bar r_S \coloneqq \dim (\calW_S)\in \N\cup \{ \infty\}$ be the maximal {\POD} rank.
    There exist non-negative eigenvalues $(\lambda_i)_{i\in \N}$ and associated orthonormal eigenfunctions $(v_i)_{i\in \N}\subset \calW_S$ of $\cor_S$ with
    \begin{align}
        \cor_S v_i & = \lambda_i v_i, \quad \lambda_1 \geq \ldots \geq \lambda_{\bar r_S} > \lambda_{\bar r_S+1}=\ldots=0,\\
        \cor_Sv &= \sum_{i=1}^\infty {\langle v_i,v \rangle}_V\,v_i,\\
        V &= \ker(\cor_S) \oplus \closure{\linspan(v_1,\ldots,v_{\bar r_{S}})}= \ker(\cor_S) \oplus \calW_S.\label{eq:spectradecomp}
    \end{align}
    The first $r\in \{1,\ldots,\bar r_S \}$ eigenfunctions $(v_i)_{i=1}^{r}$ solve \eqref{eq:POD}, i.e., $V_r=\closure{\linspan\{v_1,\ldots, v_r \}}$ and we have $V_{\bar r_S}=\calW_S$.
\end{theorem}

\begin{proof}
    This is a combination of \cite[Lemma 1.13, Theorem 1.15]{GubV17} and the spectral theorem for compact, self-adjoint operators.
\end{proof}

In Theorem~\ref{theo:POD}, we defined $V_r$ as the closure of the span to include the case $\bar r_S=\infty$ for theoretical considerations. Since we are in an adaptive setting, we are interested in how to refine a {\POD} basis. There can be two ways to extend the {\POD} basis. On the one hand, (if a finite {\POD} rank $r$ is chosen) one may increase $r\leq \bar r_S$. For the theoretical analysis presented here, we consider the choice $r= \bar r_S$. However, in practice, other choices for $r$ are possible, and the algorithm below can be adapted (see Remark~\ref{rem:choice_r}). On the other hand, one may update or enlarge the snapshots set $S$. We collect some properties of the operator $\cor_S$ in view of enlarging the snapshot set.
\begin{lemma}\label{lem:helpPOD}
    Let $S\subset S_+\subset L^2(0,T;V)$ be two snapshot sets. It holds
    \begin{enumerate}
        \item [\em a)]  $\ker(\cor_{S_+})\subset \ker(\cor_S)$;
        \item [\em b)]  $\closure{\range(\cor_{S})}\subset \closure{\range(\cor_{S_+})}$, i.e., $\calW_S\subset\calW_{S_+}$;
    \end{enumerate}
\end{lemma}

\begin{proof}
    \begin{enumerate}
        \item [a)] Let $v\in \ker(\cor_{S_+})$. Then by \eqref{eq:cor_op}, $0=\cor_{S_+}v=\cor_{S}v+\cor_{S_+\setminus S}v$ and hence
        \begin{equation}\nonumber
            0=\langle \cor_{S}v,v\rangle_{V}+\langle \cor_{S_+\setminus S}v,v\rangle_{V}
        \end{equation}
        Since both $\cor_{S}$ and $\cor_{S_+\setminus S}$ are non-negative operators according to Theorem~\ref{theo:POD}, it holds $ \cor_{S}v=0$ and thus $v\in \ker(\cor_S)$.
        \item [b)] Part a) is equivalent to $\ker(\cor_{S})^\perp\subset \ker(\cor_{S_+})^\perp$ and it holds $\ker(\cor_{S})^\perp=\closure{\range(\cor_{S})}$ ($\ker(\cor_{S_+})^\perp=\closure{\range(\cor_{S_+})}$), since $\cor_S$ ($\cor_{S_+}$) is self-adjoint by Theorem~\ref{theo:POD}.
    \end{enumerate}
\end{proof}

The following lemma provides a criterion to determine whether an enlarged snapshot set contains additional information that yields a \POD$ $ basis of higher maximal rank.

\begin{lemma}\label{lem:rank_lemma}
    Let $S\subset S_{+}\subset L^2(0,T;V)$ be two snapshot sets and $\bar r_S=\dim (\calW_S)\in \N$. Then, $\bar r_S<\bar r_{S_+}$, if and only if there exists $\tilde s\in S_{+}\setminus S $ with 
    \begin{equation}\label{eq:stilde}
        \norm{(I-\Pi_{\calW_S}^V)\tilde s}{L^2(0,T;V)}>0.
    \end{equation}
\end{lemma}

\begin{proof}
    On the one hand, let $|(I-\Pi_{\calW_S}^V)s|_{L^2(0,T;V)}=0$ for all $s\in S_{+}\setminus S $. This implies $s(t)\in \calW_S$ for almost all $t\in(0,T)$.  Then also 
    \begin{align*}
        \cor_{S_+}v=\cor_{S}v+\sum_{s\in S_+\setminus S}\int_{0}^T\langle s(t),v\rangle_{V}s(t) \ \mathrm dt\in  \calW_S,
    \end{align*}
    since the linear subspace $\calW_S$ is closed by definition in \eqref{eq:closureW}. Hence, we have $\range(\cor_{S_+})\subset \calW_S$. Taking the closure leads together with Lemma~\ref{lem:helpPOD}-b) to
    \begin{align*}
         \closure{\range(\cor_{S})}\subset \closure{\range(\cor_{S_+})}\subset\calW_S\coloneqq \closure{\range(\cor_{S})}
    \end{align*}
    Hence, $\closure{\range(\cor_{S})}= \closure{\range(\cor_{S_+})}$. This results in $\bar r_{S}=\bar r_{S_+}$ by definiton of the maximal rank in Theorem~\ref{theo:POD}.\hfill\\
    On the other hand, assume the existence of $\tilde s\in S_{+}\setminus S$ with  \eqref{eq:stilde}. W.l.o.g., we assume that $\tilde s\in S_{+}\setminus S$ is unique with this property. For $v\in \ker(\cor_S)$, it holds
    \begin{align*}
        \cor_{S_+}v = \cor_{S}v+\sum_{s\in S_+\setminus S}\int_0^T{\langle s(t),v\rangle}_V\,s(t)\dt=\int_0^T{\langle \tilde s(t),v\rangle}_V\,\tilde s(t)\dt.
    \end{align*}
    Since $\cor_S$ is self-adjoint, it holds $v\in \ker(\cor_S)=\range(\cor_S)^\perp=\closure{\range(\cor_S)}^\perp=\calW_S^\perp$.Then, it follows for $\tilde s(t)=\Pi_{\calW_S}^V\tilde s(t)+(I-\Pi_{\calW_S}^V)\tilde s(t)$ for almost all $t\in (0,T)$
    \begin{equation}\label{eq:helpPOD}
        {\langle\cor_{S_+}v, v\rangle}_V=\int_0^T\big|{\langle (I-\Pi_{\calW_S}^V)\tilde s(t),v\rangle}_V\big|^2\dt.
    \end{equation}
    Note that $\ker(\cor_S)$ is not trivial, because otherwise we would have by \eqref{eq:spectradecomp}, that $\calW_S = V$, which contradicts \eqref{eq:stilde}.
    Hence, we can additionally choose $v\neq 0$ and \eqref{eq:helpPOD} implies $\langle\cor_{S_+}v, v\rangle_V>0$, by the assumption \eqref{eq:stilde} on $\tilde s$. Hence, $\cor_{S_+}v\neq 0$ and therefore $v\notin \ker(\cor_{S_+})$.
    Together with the fact that $\ker(\cor_{S_+})\subset \ker(\cor_{S})$ by Lemma~\ref{lem:helpPOD}-a), we get $\ker(\cor_{S_+})\subsetneq \ker(\cor_{S})$. Hence, (with a similar argument as in the proof of Lemma~\ref{lem:helpPOD}-b)   $\range(\cor_{S})\subsetneq \range(\cor_{S_+})$  and therefore $\bar r_S < \bar r_{S_+}$.
\end{proof}

Once a space $V_r$ is constructed, we compute $u_1,\ldots,u_r$ via \eqref{eq:Ur_construction} and perform a Gram-Schmidt orthonormalization to obtain the space $U_r$ as in \eqref{eq:ONB}.
\subsection{Convergence analysis}

We outline the adaptive optimization in Algorithm~\ref{alg:ROM_OPT}. Each iteration $k\in \N$ involves three major building blocks, which are described in the following:
\begin{enumerate}
    \item Error certification and termination check ({Lines}~\ref{algo:line:FOMca}-\ref{algo:line:errorestend}): for the current iterate $u_k$, the lower and upper \emph{a posteriori} error estimates in \eqref{eq:error_est} are computed. As noted in Remark~\ref{rem:apost} the evaluation of the bounds need the computation of the {\FOM} quantities $y(u_k)$ and $p(y(u_k))$;
    \item Construction of the reduced spaces ({Lines}~\ref{algo:line:newsansphots}-\ref{algo:line:newsansphotsend}): if the iterate is not sufficiently accurate, better reduced spaces $V_{r_{k+1}}$ and $U_{r_{k+1}}$ are constructed via {\POD} (cf. Section~\ref{subsec:POD}) with the snapshot set $S_{k+1}=S_k\cup\{y(u_k),p(y(u_k))\}$, where $S_k$ is the snapshot set used in the previous iteration. Thus, the {\FOM} quantities to evaluate the error estimator are reused as new data to update the reduced spaces;
    \item Solution of the control- and state-reduced {\OCP} ({Line}~\ref{algo:line9}): once new reduced spaces are constructed, the control- and state-reduced {\OCP} \eqref{eq:fullROM_OCP} (see also \eqref{eq:fullROM_OCP_dis}) is solved to obtain the next iterate $u_{k+1}=\bar u^{r_{k+1}}$. Since this step is independent of the \FOM, it can be performed very efficiently if the dimension of $V_{r_{k+1}}$ and $U_{r_{k+1}}$ is small.
\end{enumerate}

In the remainder of this section, we prove that the data selection strategy in {Line~\ref{algo:line:newsansphots} of Algorithm \ref{alg:ROM_OPT}} leads to a convergent algorithm. 
\begin{algorithm}[t!]
\caption{(Adaptive {\POD} Optimization)}\label{alg:ROM_OPT}
	\begin{algorithmic}[1]
		\Require Initial guess $u_0\in L^2(0,T;U)$, tolerance $\varepsilon\geq 0$, constants $\beta$, $\norm{C\calS}{}>0$
        \State Set $S_{0}=\emptyset$;
        \For{$k=0,1,2,...$}
            \State\label{algo:line:FOMca}Compute {\FOM} state $y(u_k)$ in \eqref{eq:FOM_PDE} for $u=u_k$ and adjoint $p(y(u_k))$ in \eqref{eq:FOMadjoint} for \phantom{123}\hspace{-1.4mm} $y=y(u_k)$; 
            \State Compute the upper and {(optionally)} lower bound $\Db(u_k),\Da(u_k)$ in \eqref{eq:error_est} using \phantom{123}\hspace{-1.4mm} $y(u_k)$ and  $p(y(u_k))$;
            \If{$\Db(u_k)=\tfrac{1}{\beta}{\norm{\nabla \hat J(u_k)}{L^2(0,T;U)}}\leq \varepsilon$}
                 \Return $u_k$; \label{algo:line:errorest}
            \EndIf\label{algo:line:errorestend}
            \State\label{algo:line:newsansphots}Construct $V_{r_{k+1}}$ via {\POD} according to \eqref{eq:POD} for $S_{k+1}=S_{k}\cup \{y(u_k),p(y(u_k)\}$ for \phantom{123}\hspace{-1.4mm} $r_{k+1}=\bar r_{S_{k+1}}$;
            \State\label{algo:line:newsansphotsend}Construct $U_{r_{k+1}}$ according to \eqref{eq:Ur_construction} and \eqref{eq:ONB} by orthonormalization;
            \State\label{algo:line9}Solve control- and state-reduced {\OCP} \eqref{eq:fullROM_OCP} with $V_{r_{k+1}}$, $U_{r_{k+1}}$ for $u_{k+1}=\bar u^{r_{k+1}}$;
        \EndFor
	\end{algorithmic}
\end{algorithm}
We need two preparatory results.
\begin{lemma}[{\ROM} convergence]\label{lem:ROMconv}
    Let $(v_r)_{r\in \N}$ be an orthonormal basis of the separable space $V$ and let $V_r=\linspan(v_1,\ldots,v_r)$. Then, we have $\bar y^r \to \bar y$, $\bar p^r\to \bar p$ in $W(0,T)$ and $\bar u^r\to \bar u$ in $L^2(0,T;U)$ as $r\to \infty$.
\end{lemma}
\begin{proof}
    Since $\bar u^r=\hat u^r$ by Lemma~\ref{lem:equival}, we can consider the convergence for the state-reduced {\OCP} \eqref{eq:ROM_OCP}. This follows as in the proof of \cite[Theorem 3.11]{azmi2025stabilizationparabolictimevaryingpdes} with the only modification that in \cite{azmi2025stabilizationparabolictimevaryingpdes} controls with values in a finite-dimensional space $U$ have been considered. However, this does not change the proof.
\end{proof}
And we need the following form of an interpolation property of the reduced model.
\begin{lemma}[{\OCP} interpolation property]\label{lem:interpol}
    If $y(\bar u^r), p(y(\bar u^r))\in W^r(0,T)$, then $\bar u^r=\bar u$.
\end{lemma}
\begin{proof}
    If $y(\bar u^r), p(y(\bar u^r))\in W^r(0,T)$, then we have by unique solvability of the {\ROM} state and {\ROM} adjoint equation, that $y(\bar u^r)=y^r(\bar u^r)=\bar y^r$ and $ 
    p(y(\bar u^r))=p^r(y^r(\bar u^r))=\bar p^r$. This together with $\bar u^r=\hat u^r$ solving \eqref{eq:halfrom_opcond}, implies
    \begin{align*}
    0=\lambda \bar u^r+ \calJ_{U}^{-1} B'\bar p^r=\lambda \bar u^r+ \calJ_{U}^{-1} B'
    p(y(\bar u^r))
    \end{align*}
    Hence, $\bar u^r$ solves the {\FOM} optimality condition \eqref{eq:FOM_optcond} and by uniqueness it holds $\bar u^r=\bar u$.
\end{proof}
Now we can state a convergence result for Algorithm~\ref{alg:ROM_OPT}.
\begin{theorem}[Convergence of Algorithm~\ref{alg:ROM_OPT}]\label{theo:convergence_algo}
    If $\varepsilon =0$, then {\em Algorithm~\ref{alg:ROM_OPT}} constructs a sequence $(u_k)_{k\in \N}$ with $u_k \to \bar u$ as $k\to \infty$. For $\varepsilon>0$, {\em Algorithm~\ref{alg:ROM_OPT}} terminates after a finite number of steps $k^*\in \N$ with output $u_{k*}$ satisfying the {\FOM} optimality condition up to tolerance $\beta \varepsilon$, i.e.,
    \begin{equation}\label{eq:term1}
        \norm{\nabla \hat J(u_{{k^*}})}{L^2(0,T;U)}\leq \beta \varepsilon.
    \end{equation}
    and the estimate
    \begin{equation}\label{eq:term2}
        0\leq \Da( u_{k^*}) \leq \norm{\bar u- u_{k^*}}{L^2(0,T;U)}\leq \Db (u_{k^*}) \leq  \varepsilon.
    \end{equation}
\end{theorem}

\begin{proof}
    First, let $\varepsilon=0$ and suppose at step $k\in \N$, it holds $\Db(u_k)=\nicefrac{1}{\beta}\,|\nabla \hat J(u_k)|_{L^2(0,T;U)}>0$ in line \ref{algo:line:errorest} of Algorithm~\ref{alg:ROM_OPT}. Necessarily we have $\bar r_{S_k}<\infty$. Then, by the lower bound in \eqref{eq:error_est} it holds $\bar u \neq u_k = \bar u^{r_k}$. Therefore, we have by Lemma~\ref{lem:interpol}, that $y(u_k)\notin W^{r_k}(0,T)$ or $p(y(u_k))\notin W^{r_k}(0,T)$. Hence, the snapshot set $S_{k+1}=S_k \cup \{y(u_k), p(y(u_k))\}$ defined in line \ref{algo:line:newsansphots} of Algorithm~\ref{alg:ROM_OPT} contains new information. If $\bar r_{S_{k+1}}=\infty$, we directly obtain $\bar r_{S_k}<\bar r_{S_{k+1}}$ and also $u_{k+1}=\bar u$, i.e., convergence after a finite number of steps for $\varepsilon=0$. If $\bar r_{S_{k+1}}< \infty$, there exists $\tilde s\in \{y(u_k), p(y(u_k))\}$ with \eqref{eq:stilde}, since else it would hold $y(u_k)\in L^2(0,T;V_{r_k})$ and $p(y(u_k))\in  L^2(0,T;V_{r_k})$, due to ${V}_{r_k}=\calW_{S_k}$ for ${r_k}=\bar r_{S_k}$. Then by uniqueness of the {\ROM} state equation \eqref{eq:ROM_PDE} in $L^2(0,T;V_{r_k})$ one would obtain again $y(u_k)\in W^{r_k}(0,T)$ and $p(y(u_{r_k}))\in  W^{r_k}(0,T)$, which we excluded above. Applying Lemma~\ref{lem:rank_lemma} for $S= S_k$ and $S_+=S_{k+1}$ implies $\bar r_{S_k}<\bar r_{S_{k+1}}$. By iterating this procedure, we construct an orthonormal basis of $V$ and Lemma~\ref{lem:ROMconv} ensures convergence, since $r_k\to \infty$ as $k\to \infty$. Now $u_k\coloneqq \bar u^{r_k}\to \bar u$ and \eqref{eq:error_est} imply $\nabla \hat J(u_k)\to 0$ as $k\to \infty$. Hence, for $\varepsilon>0$, the termination criterion in line \ref{algo:line:errorest} of Algorithm~\ref{alg:ROM_OPT} gets triggered after a finite number of steps and ensures \eqref{eq:term1} and \eqref{eq:term2} using \eqref{eq:error_est}.
\end{proof}

\begin{remark}\label{rem:choice_r}
    Algorithm~\ref{alg:ROM_OPT} can be modified or improved in various ways. For example, instead of choosing the maximal {\POD} rank in line \ref{algo:line:newsansphots} of Algorithm~\ref{alg:ROM_OPT}, one can make a choice based on an energy criterion (cf. \cite[Theorem 1.8]{GubV17}), i.e., choose $r_k\in \N$ such that
    {
    \begin{align}\label{eq:energy}
       \frac{\sum_{i=1}^{r_k}\lambda_i}{\sum_{i=1}^{\bar r_S}\lambda_i}=\frac{\sum_{i=1}^{r_k}\lambda_i}{\sum_{s\in S}|s|_{L^2(0,T;V)}^2}\leq 1-\varepsilon_{\mathsf{POD}},
    \end{align}
    for some tolerance $\varepsilon_{\mathsf{POD}}\in [0,1)$,} and include another inner loop to adjust $r_k\leq \bar r_{S_k}$ around line \ref{algo:line9} of Algorithm~\ref{alg:ROM_OPT}. Since the estimators in \eqref{eq:error_est} are expensive to evaluate, it might not be of advantage to use them in the inner loop for adjusting $r$. Instead, one could combine the certification process for the inner loop with the cheap-to-evaluate (i.e., independent of {\FOM} calculations) error estimators presented, e.g., in \cite{kartmann2024certifiedmodelpredictivecontrol,karcher2018certified,ali2020reduced} (see also Remark \ref{rem:apost}). {Note that the snapshot set in Algorithm~\ref{alg:ROM_OPT} grows monotonically, which may
lead to high storage requirements and increased {\POD} costs for large-scale problems.
As a remedy, one may employ hierarchical approximate {\POD} techniques ({\HAPOD},
cf.~\cite{himpe2018hierarchical}), where the {\POD} basis is updated incrementally as new
data becomes available, avoiding repeated recomputation on an ever-growing snapshot set.
Moreover, it may be advantageous to develop criteria for removing outdated snapshots from
$S_{k+1}$.} In view of {\MOR}, i.e., to obtain a fast approximation of $\bar u$, Algorithm~\ref{alg:ROM_OPT} should always be run with a tolerance $\varepsilon>0$. For parabolic differential operators with analytic coefficients, one can expect an exponential decay of the eigenvalues $\lambda_i$ in Theorem~\ref{theo:POD} and thus Algorithm~\ref{alg:ROM_OPT} to converge fast if $\varepsilon>0$. Last, in a numerical implementation, it is advised to use the previous iterate as a warm start for the next inner optimization in Algorithm~\ref{alg:ROM_OPT}.
\end{remark}
\begin{remark}[Extension to nonlinear problems]\label{rem:extension_to_nonlinear}
{
We believe that a transfer of the results obtained in this work to nonlinear {\PDEs} is feasible. 
A key requirement for our approach is that the control enters the {\PDE} linearly and the cost 
functional is quadratic in the control. Under these assumptions, the same arguments as above 
show that the first-order optimality conditions imply that every stationary point lies in the 
reduced control space. Moreover, assuming a second-order sufficient optimality condition, 
one can derive analogous (local) stability estimates, which are essential for the a~posteriori 
error analysis.
}
\end{remark}
\section{Numerical Results}\label{sec:numexps}

\subsection{Model problem}\label{subsec:problem}

Let $\Omega=(0,1)^2$, $U=H=L^2(\Omega)$, and $V=H^1(\Omega)$. In the numerical experiments, we consider the following {\OCP}
\begin{subequations}\label{subEq6}
    \begin{align}
        \label{eq:OCP_concrete}
        \min J(y,u)\coloneqq \int^{T}_{0}\tfrac{1}{2}\norm{y(t)-y_d(t)}{H}^2+\tfrac{\beta}{2} \norm{u(t)}{U}^2\dt
    \end{align}
    subject to $u\in L^2(0,T;U)$ and $y\in W(0,T)$ solving
    \begin{align}\label{eq:FOM_PDE_concrete}
        \begin{cases}
            \ddt y(t)+\Delta y(t)+(2+\sin(4\pi t))y(t) = u(t)\in V'  &\text{in } (0,T),\\
            y(0)=y_0 & \text{in }H
        \end{cases}
    \end{align}
\end{subequations}
for $T=1$ and different parameters of $\beta>0$. The target state and initial value are given for $t\in [0,T], \bx \in \Omega$ by
\begin{align*}
     y_d(t,\bx) = \tfrac{1}{6}\,\sin(2\pi \bx_1t)\sin(2\pi \bx_2t)\exp(2\bx_1), \quad y_0(\bx) = 0.
\end{align*}
It can be shown that this problem satisfies assumption \textbf{(A1)}, \textbf{(A2)}. Note that, since $U=H$ and $B$ being the embedding from $U$ into $V'$, the optimality condition takes the form
\begin{align*}
    \bar u(t) = -\tfrac{1}{\beta}\,\bar p(t)\quad \text{for almost all } t\in (0,T).
\end{align*}
Thus, $\bar u(t) \in V$ a.e., and we can use the same reduced space for the state space and the control space, $U_r = V_r$. Further, a lower bound for the constant in \eqref{eq:error_est} can be derived from $\norm{C\calS}{}\leq \eta_V=1$. This follows since $C$ is the embedding from $W(0,T)$ into $L^2(0,T;H)$ for the choice of $V$, and $H$ above, and the tracking term in \eqref{eq:OCP_concrete}. Then, a standard a priori estimate leads to $\norm{\calS}{}=\eta_V=1$ for the {\PDE} in \eqref{eq:FOM_PDE_concrete}. Thus, to evaluate the error estimator $\Da(\tilde u)$ from \eqref{eq:error_est}, we use the constant $c_\mathsf a>0$ with
\begin{equation}\label{eq:lower constant}
   c_\mathsf a\coloneqq  \tfrac{1}{\beta+1}\leq \tfrac{1}{\beta+\norm{C\calS}{}^2}.
\end{equation}


\subsection{Discretization and algorithmic setup} \label{sec:algsetup}

We compare three algorithms:
\begin{enumerate}
    \item {\FOM}: {\FOM} Barzilai-Borwein (\BB) gradient descent (cf. \cite{azmi2022convergence});
    \item Control-{\ROM}: adaptive {\POD-\ROM} gradient descent with state and control reduction, i.e., Algorithm~\ref{alg:ROM_OPT}, where \eqref{eq:fullROM_OCP} is solved by {\BB} gradient descent;
    \item {\ROM}: adaptive {\POD-\ROM} gradient descent with state-space reduction but without control reduction, i.e., Algorithm~\ref{alg:ROM_OPT} with \eqref{eq:halfrom_opcond} in place of \eqref{eq:fullROM_OCP} in line \ref{algo:line9} of Algorithm~\ref{alg:ROM_OPT}, solved by {\BB} gradient descent.
 \end{enumerate}
For the {\FOM} discretization of the problem from Section~\ref{subsec:problem}, we use piecewise linear finite elements in space with $N=10201$ degrees of freedom and $K=100$ time steps in an implicit Euler scheme for time discretization. {The {\FOM} discretization parameters are chosen to represent a moderate-scale problem and
to demonstrate a meaningful speed-up of the {\ROM} methods. The {\POD} basis size is
selected using an energy criterion (see \eqref{eq:energy} in Remark~\ref{rem:choice_r})}, with a tight tolerance
$\varepsilon_{\mathsf{POD}} = 1 - 10^{-12}$, which effectively corresponds to using the
maximal POD rank. For the optimization, we use the stopping criterion $|\nabla \hat J(u_k)|_{L^2(0,T;U)}\leq \tau$ as a stopping criterion for both the {\FOM} and {\ROM} methods for several parameters $\tau$ chosen below. This corresponds to the stopping criterion used in Algorithm~\ref{alg:ROM_OPT} for $\varepsilon=\nicefrac{\tau}{\beta}$ (see also \eqref{eq:term1}). The initial guess $u_0$ for the optimization is chosen to be $u_0\equiv 0$. To verify the error estimator statements and compare the sharpness w.r.t. the true error, we compute the true solution $\bar u$ of the {\FOM-\OCP} using the {\BB} method with a tolerance of $\tau = 10^{-12}$ for the gradient norm. We provide the source code of the numerical experiments in \cite{code}, which were performed on a MacBook Pro 2020 with a 2,3 GHz Quad‑Core Intel Core i7.
%
\subsection{Comparison of the state-reduced {\OCP} and the control- and state-reduced {\OCP} }
%
In the following, we want to numerically verify the equivalence in Lemma~\ref{lem:equival}, that is we numerically solve the state-reduced {\OCP} \eqref{eq:ROM_OCP} and the control- and state-reduced {\OCP} \eqref{eq:fullROM_OCP} for a fixed reduced space $V_r$ and compare the results in terms of approximation quality and computational performance for a gradient tolerance of $\tau =10^{-12}$. To this end, we construct the space $V_r$ via {\POD} for the snapshot set $S=\{y(u_0),p(u_0)\}$, leading to a basis of size $r=17$. In Table~\ref{tab:rom_rom_comparison}, we compare the errors in state, adjoint, and control between the two models for different values of $\beta\in \{10^{-1},10^{-2}, 10^{-3}, 10^{-4} \}$ and the speed-up of \eqref{eq:fullROM_OCP} compared to \eqref{eq:ROM_OCP}. The errors are defined as
\begin{equation}\nonumber
    \text{err}_y\coloneqq \norm{\hat y^r - \bar y^r}{L^2(0,T;H)},\quad \text{err}_p\coloneqq \norm{\hat p^r - \bar p^r}{L^2(0,T;H)},\quad \text{err}_u\coloneqq \norm{\hat u^r - \hat u^r}{L^2(0,T;U)}.
\end{equation}
We observe in Table \ref{tab:rom_rom_comparison} that the control errors err$_u$ between the two reduced models are in the range of order $-16$ to $-9$ for decreasing choices of $\beta$, while the errors in the state and the adjoint are in the range of order $-17$ to $-11$ correspondingly. This confirms the equivalence result in Lemma~\ref{lem:equival}. Further, we observe a speed-up of a factor of $9$ for the control- and state-reduced {\OCP} \eqref{eq:fullROM_OCP} compared to the state-reduced {\OCP} \eqref{eq:ROM_OCP}, which is independent of $\beta$. 

Thus, control reduction is of advantage in terms of computational performance, while the approximation quality is affected only marginally.
\begin{table}[ht!] 
     \scriptsize
	\centering 
    \caption{Errors and speed-ups between state-reduced {\OCP} \eqref{eq:ROM_OCP} and control- and state-reduced {\OCP} \eqref{eq:fullROM_OCP} for different values of $\beta$}
	\label{tab:rom_rom_comparison}
	\begin{tabular}{lcccc}\toprule
		$\beta$ & $10^{-1}$ & $10^{-2}$ & $10^{-3}$& $10^{-4}$  \\ 
        \midrule
		$\text{err}_u$ & $3.73$e--$16$ & $6.96$e--$13$ & $5.85$e--$10$ & $5.29$e--$9$ \\
        $\text{err}_y$ & $9.43$e--$17$ & $2.37$e--$13$ & $1.17$e--$11$ & $1.87$e--$11$ \\
        $\text{err}_p$ & $4.21$e--$17$ & $2.54$e--$15$ & $6.29$e--$13$ & $7.26$e--$13$ \\
        speed-up & $9.5$ & $9.4$ & $9.2$ & $9.5$ \\
		\bottomrule
	\end{tabular}
\end{table}
%
\subsection{Comparison of the adaptive {\POD} optimization and the \FOM}
%
Next, we approximately solve the {\FOM-\OCP} using the adaptive {\ROM} method in Algorithm~\ref{alg:ROM_OPT}, and numerically verify the error estimator statement \eqref{eq:error_est} from Corollary~\ref{cor:error_est_control}, and the convergence of Algorithm~\ref{alg:ROM_OPT} in Theorem~\ref{theo:convergence_algo}, and compare the methods in terms of computational performance for a gradient tolerance of $\tau =10^{-8}$.
\begin{figure}
    \centering
    \begin{subfigure}[b]{0.32\textwidth}
         \centering
          \includegraphics[width=1\linewidth]{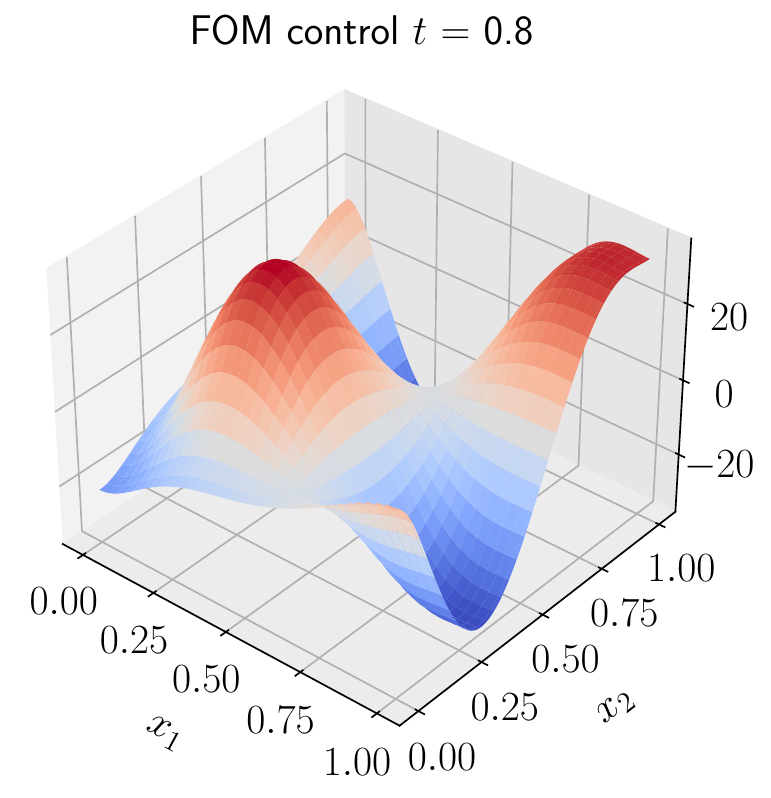}
         \caption{}
    \end{subfigure}
     \begin{subfigure}[b]{0.32\textwidth}
         \centering
          \includegraphics[width=1\linewidth]{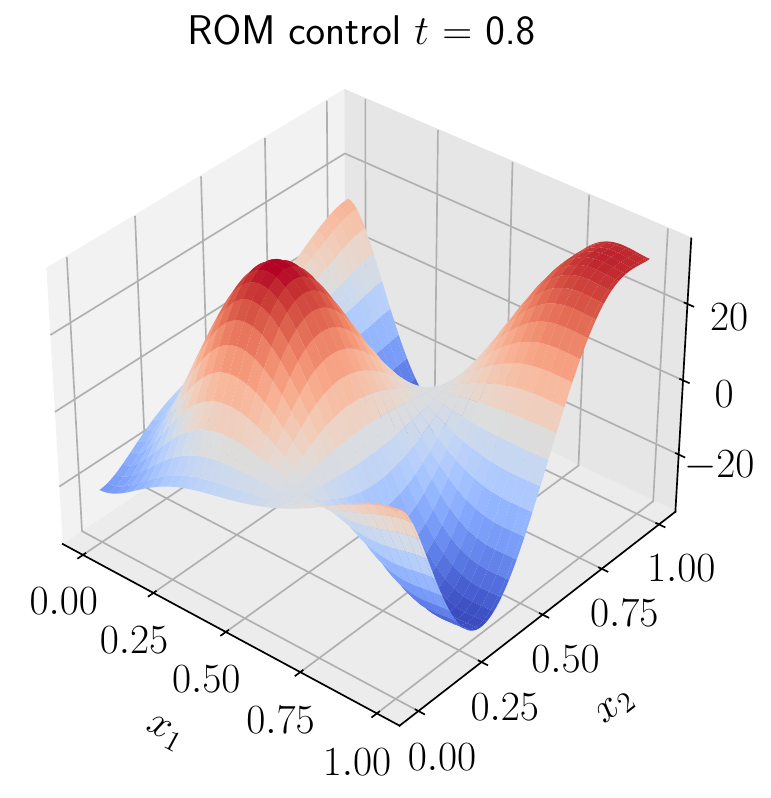}
         \caption{}
    \end{subfigure}
     \begin{subfigure}[b]{0.32\textwidth}
         \centering
          \includegraphics[width=1\linewidth]
          {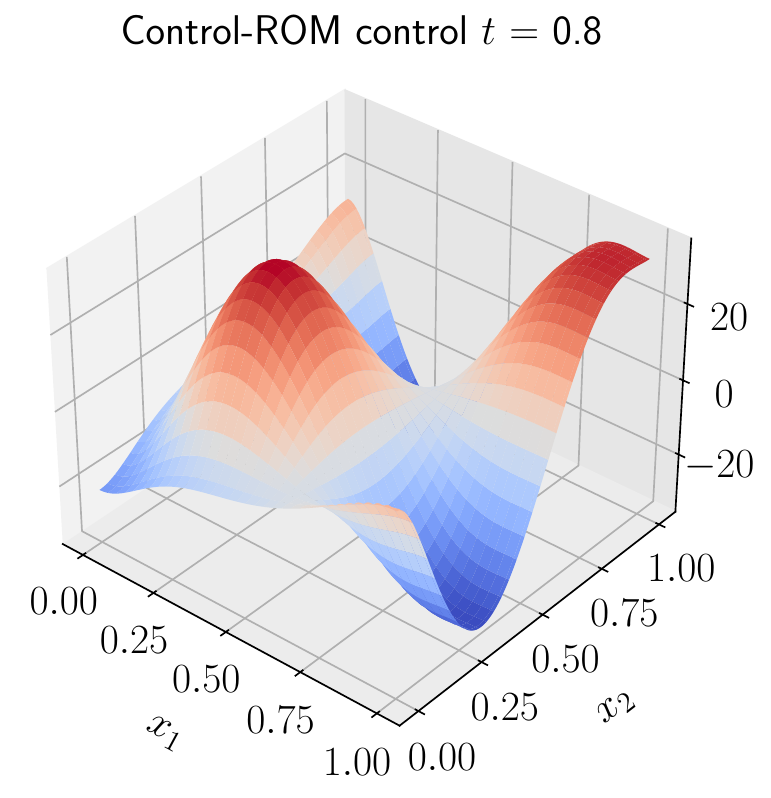}
         \caption{}
    \end{subfigure}

    \caption{Snapshots of the optimal controls generated by each algorithm for $\beta =10^{-4}$}
    \label{fig:controls}
\end{figure}
\begin{table}[ht!] 
    \scriptsize
	\centering 
    \caption{Performance comparison of \FOM, \ROM, and Control-\ROM}
	\label{tab:perfromacne}
	\begin{tabular}{lcccccc}\toprule
		$\beta =10^{-1}$  & time [s] & speed-up & $k$& $\Da(\bar u^r )$ &$e_u(\bar u^r )$ &$\Db(\bar u^r )$  \\ 
        \midrule
		\FOM &  117  &  - & 10 & -& $3.33$e--$8$&- \\
        \ROM   & 33 & 3.6 & 2 &$4.54$e--$9$& $2.49$e--$8$ & $5.00$e--$8$ \\
        Control-\ROM   & 30 & 3.9 & 2 &$4.55$e--$9$& $2.49$e--$8$ & $5.00$e--$8$\\
        \midrule
        $\beta =10^{-2}$  & time [s] & speed-up & $k$& $\Da(\bar u^r )$ &$e_u(\bar u^r )$ &$\Db(\bar u^r )$  \\ 
        \midrule
		\FOM   & 306  &  - & 26 & -& $1.90$e--$7$&- \\
        \ROM   & 35 &8.6 & 2 & $6.47$e--$9$&$1.97$e--$7$& $6.53$e--$7$\\
        Control-\ROM   & 34 & 9 & 2 & $6.47$e--$9$&$1.97$e--$7$& $6.53$e--$7$\\
        \midrule
        $\beta =10^{-3}$  & time [s] & speed-up & $k$& $\Da(\bar u^r )$ &$e_u(\bar u^r )$ &$\Db(\bar u^r )$  \\ 
        \midrule
		\FOM   & 699  &  - & 76 & -& $2.76$e--$6$ &- \\
        \ROM   & 68 & 10.3 & 3 &$9.18$e--$9$&$9.09$e--$6$&$9.19$e--$6$\\
        Control-\ROM  & 56 & 12.5 & 3 &$7.93$e--$9$&$7.09$e--$6$&$7.94$e--$6$\\
        \midrule
        $\beta =10^{-4}$  & time [s] & speed-up & $k$& $\Da(\bar u^r )$ &$e_u(\bar u^r )$ &$\Db(\bar u^r )$  \\ 
        \midrule
		\FOM &  2048  &  - & 208 & -& $1.48$e--$5$ &- \\
        \ROM   & 99 & 20.6 & 4 & $6.58$e--$9$& $5.93$e--$5$ &$6.58$e--$5$ \\
        Control-\ROM   & 59 & 34.2 & 4 &$9.66$e--$9$&$8.82$e--$5$&$9.66$e--$5$\\
		\bottomrule
	\end{tabular}
\end{table}
In Figures~\ref{fig:controls}, we plot the optimal controls for $t=0.8$ and $\beta=10^{-4}$ for all three methods. The {\FOM}, {\ROM}, and Control-{\ROM} solutions are visually indistinguishable from the {\FOM} solution. {This is also reflected in the true control errors $e_u(\bar u^r )$ reported in Table~\ref{tab:perfromacne} ranging from $10^{-8}$ to $10^{-5}$ for $\beta$ from $10^{-1}$ to $10^{-4}$}. In Table~\ref{tab:perfromacne}, we also compare the methods in terms of performance for $\beta \in \{10^{-4}, 10^{-3}, 10^{-2}, 10^{-1} \}$. We observe that the {\ROM} and Control-{\ROM} methods achieve speed-ups up to $20.6$ and $34.2$, respectively, compared to the {\FOM} method. The speed-up increases with decreasing $\beta$, which can be explained as follows. Since we use a gradient method, the cost of one {\FOM} iteration consists mainly of one gradient evaluation, which needs one {\FOM} state and one {\FOM} adjoint solve. On the other hand, the cost of one {\ROM} or Control-{\ROM} iteration is mainly determined by one {\FOM} state and one {\FOM} adjoint solve to evaluate the error estimator in line \ref{algo:line:errorest} of {Algorithm \ref{alg:ROM_OPT}} in addition to the construction of the reduced spaces and the solution of the reduced {\OCP}. Thus, the cost of one {\FOM} iteration is higher than the cost of one (Control-){\ROM} iteration, and speedups can only be expected if the (Control-){\ROM} achieves a reduction in outer iterations $k$. This is the case for decreasing $\beta$, since the problem becomes more ill-conditioned. We observe that the Control-{\ROM} method achieves a higher speed-up than the {\ROM} method, since the reduced problem for the inner problem for the Control-{\ROM} is completely independent of the {\FOM} dimension. These results indicate that control reduction pays off numerically. Also, the true error and the upper error estimators increase with decreasing $\beta$, as it is expected by the factor $\nicefrac{1}{\beta}$ in \eqref{eq:error_est}. Note that \eqref{eq:term2} in Theorem~\ref{theo:convergence_algo} is valid for all cases, since we have $e_u(\bar u^r )\leq\varepsilon=\nicefrac{10^{-8}}{\beta}\in\{10^{-4}, 10^{-5}, 10^{-6}, 10^{-7} \}$ for $\beta\in \{10^{-4}, 10^{-3}, 10^{-2}, 10^{-1} \}$.

Furthermore, we observe in Table~\ref{tab:perfromacne} that the error estimator statement \eqref{eq:error_est} from Corollary~\ref{cor:error_est_control} is valid for $\tilde u = \bar u^r$. 
This has also been verified for $\tilde u = u_k$ across all {\ROM} methods, as illustrated in Figure~\ref{fig:est}.
\begin{figure}
    \centering
    \begin{subfigure}[b]{0.41\textwidth}
         \centering
          \includegraphics[width=1\linewidth]{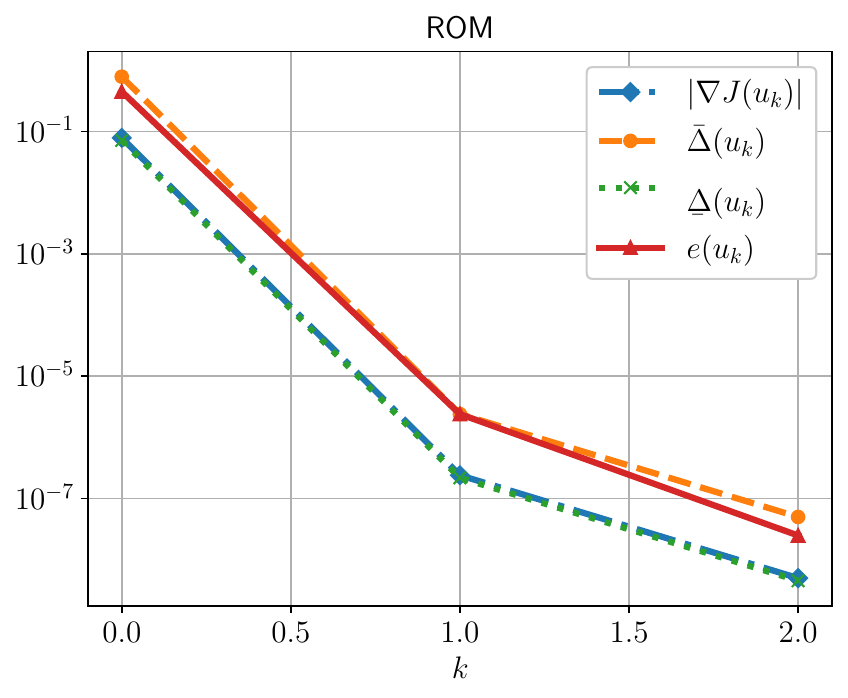}
         \caption{$\beta=10^{-1}$}
     \end{subfigure}
     \begin{subfigure}[b]{0.41\textwidth}
         \centering
          \includegraphics[width=1\linewidth]{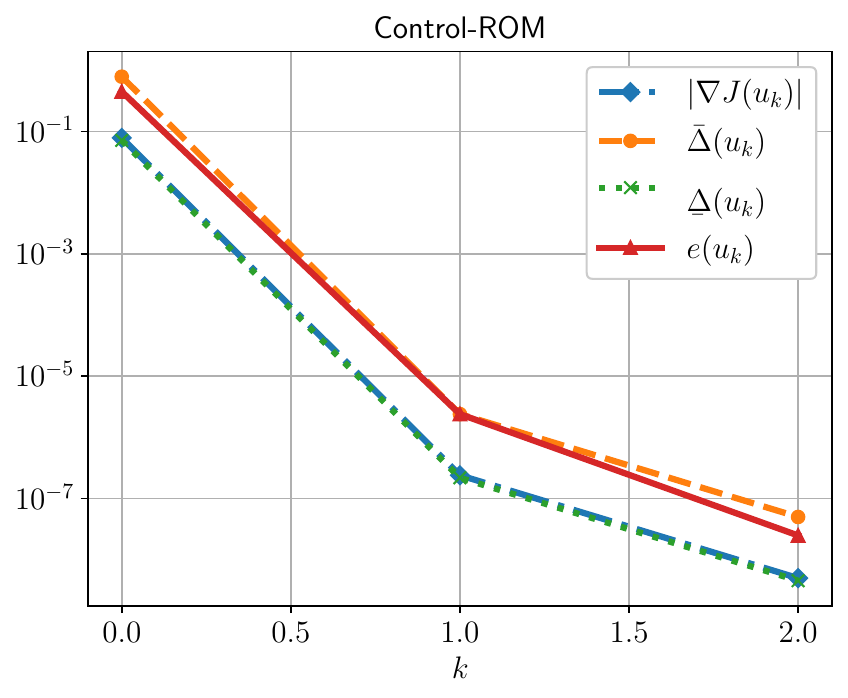}
         \caption{$\beta=10^{-1}$}
     \end{subfigure}
     \\
     \begin{subfigure}[b]{0.43\textwidth}
         \centering
          \includegraphics[width=1\linewidth]{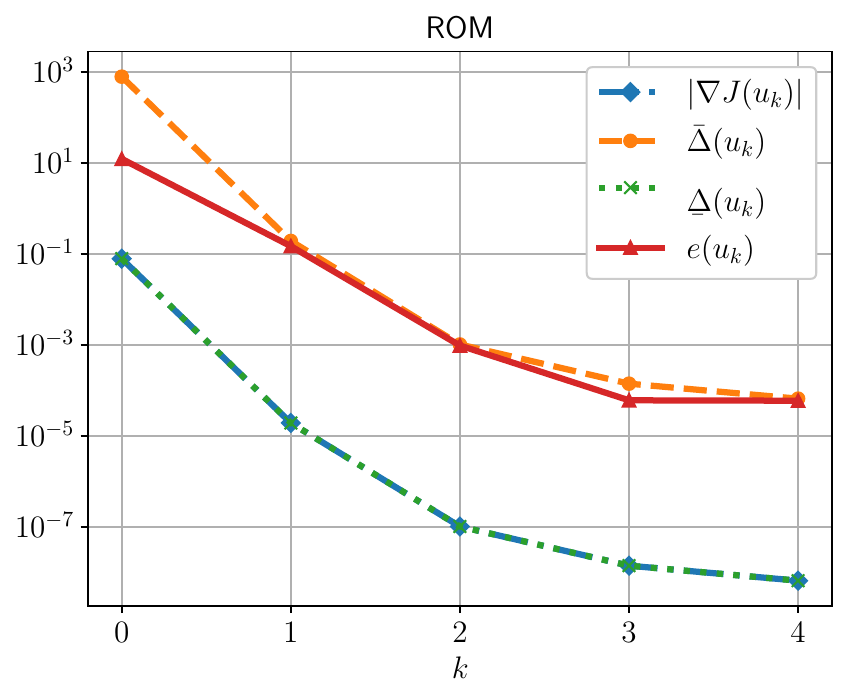}
    \caption{$\beta=10^{-4}$}
     \end{subfigure}
     \begin{subfigure}[b]{0.41\textwidth}
         \centering
          \includegraphics[width=1\linewidth]{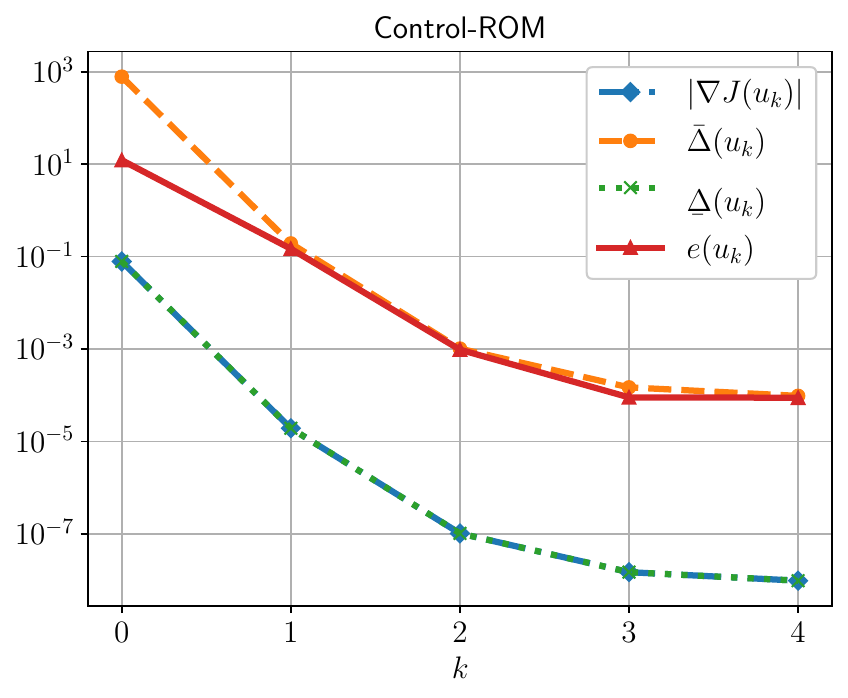}
      \caption{$\beta=10^{-4}$}
     \end{subfigure}
    \caption{The error quantities from \eqref{eq:error_est}, and $|\nabla \hat J(u_k)|_{L^2(0,T;U)}$ plotted against iteration counter $k$ for different $\beta$, and both {\ROM} (left column) and {Control-\ROM} methods (right column)}
    \label{fig:est}
\end{figure}
There the gradient norm $|\nabla \hat J(u_k)|_{L^2(0,T;U)}$, the true errors $e_u(u_k)$ and the error estimators $\Da(u_k)$ and $\Db(u_k)$ are plotted over the outer iterations $k$ of Algorithm~\ref{alg:ROM_OPT} for $\beta\in \{10^{-4}, 10^{-3}, 10^{-2}, 10^{-1} \}$. We observe that the upper estimator is sharp, whereas the lower estimator is less accurate due to the possibly coarse approximation $c_\mathsf a$ of the lower bound constant given in \eqref{eq:lower constant}. Moreover, the lower bound matches the gradient norm, i.e., $\Da(u_k)\approx |\nabla \hat J(u_k)|_{L^2(0,T;U)}$, since it holds  $\bar c\approx 1$ for $\beta\in \{10^{-4}, 10^{-3}, 10^{-2}, 10^{-1} \}$. In total, Corollary~\ref{cor:error_est_control} is numerically verified.

In {Table~\ref{tab:basisize}}, we consider the basis size $r_k$ over the outer iterations $k$ of Algorithm–\ref{alg:ROM_OPT} for $\beta\in \{10^{-4}, 10^{-3}, 10^{-2}, 10^{-1} \}$. In all cases, we observe a monotonically increasing effect of the basis size. {This is due to the accumulation of snapshot information by the snapshot selection strategy in Algorithm~\ref{alg:ROM_OPT}}. Since the initial guess $u_0$ and initial value $y_0$ are chosen to be zero, the basis size for the first iteration only contains information about the adjoint state, resulting in a smaller basis size. Then, the algorithm moves quickly towards the optimal solution, and the basis size increases and stabilizes around $r=20$ to $r=40$ for all cases.
%
%
\begin{table}[ht!] 
    \scriptsize
	\centering 
    \caption{{Evolution of the basis size $r_k$ during iterations $k$}}
	\label{tab:basisize}
	\begin{tabular}{lcccc}\toprule
        $k$ &1&2&3&4\\
        \midrule
		$\beta =10^{-1}$  &  &  & &    \\ 
        \midrule
        \ROM   & 17 & 23 & - &-\\
        Control-\ROM   & 17 & 23 & - &-\\
        \midrule
        $\beta =10^{-2}$  &  &  & &  \\ 
        \midrule
         \ROM   & 17 & 24 & - &-\\
        Control-\ROM   & 17 & 24 & - &-\\
        \midrule
        $\beta =10^{-3}$  &  &  & &   \\ 
        \midrule
        \ROM   & 17 & 27 & 31 &-\\
        Control-\ROM   & 17 & 27 & 32 &-\\
        \midrule
        $\beta =10^{-4}$  &  &  & &    \\ 
        \midrule
        \ROM   & 17 & 30 & 36 & 43\\
        Control-\ROM   & 17 & 30 & 37 &43\\
		\bottomrule
	\end{tabular}
\end{table}
%
%
\section{Conclusion}\label{sec:conclusion}
We demonstrated that, for unconstrained linear-quadratic optimal control problems, a reduction in the state variable induces a corresponding reduced structure in the optimal control. Consequently, the state-reduced \OCP\ can be equivalently reformulated as an \OCP\ reduced in both 
the state and control spaces. Furthermore, we established convergence of the reduced-order model {\ROM} towards the {\FOM} and derived \emph{a posteriori} upper and lower error bounds for the error in the optimal control. Based on these results, we proposed an adaptive optimization algorithm for the approximate 
solution of the \FOM-\OCP\ and showed its convergence both theoretically and numerically. It turned out that our {Control-\ROM} achieves a higher speed-up (compared to the {\FOM-\OCP}) than the gradient descent with only state-space reduction. {In future work, it would be interesting to extend the strategy to nonlinear \PDEs.}

\subsection*{Declarations}
\begin{itemize}
    \item \textbf{Funding}: The authors acknowledge funding by the Federal Ministry of Education and Research (grant no. 05M22VSA).
    \item \textbf{Code availability statement}: The code to reproduce the numerical experiments is available at \cite{code}.
    \item \textbf{Competing Interests}: The authors have no relevant financial or non-financial interests to disclose.
    \item \textbf{Data availability}: Data sharing is not applicable to this article as no datasets were generated or analysed during the current study.
    \item \textbf{Authors' contributions}: M.K. prepared the initial manuscript and carried out the software implementation. S.V. was responsible for supervision and funding acquisition. All authors reviewed and approved the final manuscript.
\end{itemize}
\bibliography{biblio.bib}

@article{KV08,
    author = {Kunisch, K. and Volkwein, S.},
    title = {Proper orthogonal decomposition for optimality systems},
    journal = {ESAIM: Mathematical Modelling and Numerical Analysis},
    volume = {42},
    pages = {1-23},
    year = {2008},
    doi = {10.1051/m2an:2007054}}

@article{kone2017numerical,
  title={A Numerical Approach to Possible Identification of the Noisiest Zones of a Wall Surface with a Flow Interaction},
  author={Kone, Tenon Charly and Marchesse, Yann and Panneton, Raymond and others},
  journal={Open Journal of Fluid Dynamics},
  volume={7},
  number={04},
  pages={525},
  year={2017},
  doi = {10.4236/ojfd.2017.74036},
  publisher={Scientific Research Publishing}
}

@article{code,
    author = {Kartmann, M. and Volkwein, S},
    title = {{C}ode for "{O}ptimality-{B}ased {C}ontrol {R}eduction for {I}nfinite-{D}imensional {C}ontrol {S}paces"},
    doi = {10.5281/zenodo.17356821},
    url = {https://doi.org/10.5281/zenodo.17356820},
    year = {2025}
}

@article{haasdonk2017reduced,
  title={Reduced basis methods for parametrized {PDE}s--a tutorial introduction for stationary and instationary problems},
  author={Haasdonk, Bernard},
  journal={Model reduction and Approximation: Theory and Algorithms},
  volume={15},
  pages={65},
  year={2017},
  doi = {10.1137/1.9781611974829.ch},
  publisher={Siam Philadelphia}
}

@article{qian2017certified,
  title={A certified trust region reduced basis approach to PDE-constrained optimization},
  author={Qian, Elizabeth and Grepl, Martin and Veroy, Karen and Willcox, Karen},
  journal={SIAM Journal on Scientific Computing},
  volume={39},
  number={5},
  pages={S434--S460},
  year={2017},
  doi = {10.1137/16M1081981},
  publisher={SIAM}
}

@article{keil2021non,
  title={A non-conforming dual approach for adaptive trust-region reduced basis approximation of {PDE}-constrained parameter optimization},
  author={Keil, Tim and Mechelli, Luca and Ohlberger, Mario and Schindler, Felix and Volkwein, Stefan},
  journal={ESAIM: Mathematical Modelling and Numerical Analysis},
  volume={55},
  number={3},
  pages={1239--1269},
  year={2021},
  doi = {10.1051/m2an/2021019},
  publisher={EDP Sciences}
}

@article{volkwein2011optimality,
  title={Optimality system {POD} and a-posteriori error analysis for linear-quadratic problems},
  author={Volkwein, Stefan},
  journal={Control and Cybernetics},
  volume={40},
  number={4},
  url = {https://bibliotekanauki.pl/articles/206100.pdf},
  pages={1109--1124},
  year={2011}
}

@article{afanasievadaptive,
author = {Afanasiev, Konstantin and Hinze, Michael},
year = {2001},
month = {01},
pages = {},
title = {Adaptive Control Of A Wake Flow Using Proper Orthogonal Decomposition},
volume = {216},
isbn = {978-0-8247-0556-5},
journal = {Lect. Notes Pure Appl. Math.},
doi = {10.1201/9780203904169.ch13}
}

@article{ali2020reduced,
  title={Reduced Basis Methods -- An Application to Variational Discretization of Parametrized Elliptic Optimal Control Problems},
  author={Ali, A.A. and Hinze, M.},
  journal={SIAM Journal on Scientific Computing},
  volume={42},
  number={1},
  pages={A271--A291},
  year={2020},
  doi={10.1137/18M1227147},
  publisher={SIAM}
}

@article{azmi2025stabilizationparabolictimevaryingpdes,
      title={Stabilization of Parabolic Time-Varying {PDE}s using Certified Reduced-Order Receding Horizon Control}, 
      author={Behzad Azmi and Michael Kartmann and Stefan Volkwein},
      volume={},
      number={},
      pages={},
      year={2025},
      journal={arXiv},
      doi={10.48550/arXiv.2508.16801}, 
}

@article{azmi2022convergence,
  title={On the convergence and mesh-independent property of the {B}arzilai--{B}orwein method for {PDE}-constrained optimization},
  author={Azmi, Behzad and Kunisch, Karl},
  journal={IMA Journal of Numerical Analysis},
  volume={42},
  number={4},
  pages={2984--3021},
  year={2022},
  doi = {10.1093/imanum/drab056},
  publisher={Oxford University Press}
}

@ARTICLE{8014482,
  author={Sadati, S. M. Hadi and Naghibi, S. Elnaz and Walker, Ian D. and Althoefer, Kaspar and Nanayakkara, Thrishantha},
  journal={IEEE Robotics and Automation Letters}, 
  title={Control Space Reduction and Real-Time Accurate Modeling of Continuum Manipulators Using {R}itz and {R}itz–{G}alerkin Methods}, 
  year={2018},
  volume={3},
  number={1},
  pages={328-335},
  doi={10.1109/LRA.2017.2743100}}

@article{chun2024multiscale,
  title={Multiscale optimization via enhanced multilevel {PCA}-based control space reduction for electrical impedance tomography imaging},
  author={Chun, Maria MFM and Edwards, Briana L and Bukshtynov, Vladislav},
  journal={Computers \& Mathematics with Applications},
  volume={157},
  pages={215--234},
  year={2024},
  publisher={Elsevier},
  doi = {10.1016/j.camwa.2024.01.007}}

@article{delavari2025action,
  title={Action Space Reduction Strategies for Reinforcement Learning in Autonomous Driving},
  author={Delavari, Elahe and Khanzada, Feeza Khan and Kwon, Jaerock},
  journal={arXiv},
  doi = {10.48550/arXiv.2507.05251},
  pages = {},
  volume = {},
  year={2025}
}

@article{kartmann2024adaptive,
  title={Adaptive reduced basis trust region methods for parameter identification problems},
  author={Kartmann, Michael and Keil, Tim and Ohlberger, Mario and Volkwein, Stefan and Kaltenbacher, Barbara},
  journal={Computational Science and Engineering},
  volume={1},
  number={1},
  pages={3},
  year={2024},
  doi = {10.1007/s44207-024-00002-z},
  publisher={Springer}
}

@article{kartmann2025adaptive,
  title={Adaptive Reduced Basis Trust Region Methods for Parabolic Inverse Problems},
  author={Kartmann, Michael and Klein, Benedikt and Ohlberger, Mario and Schuster, Thomas and Volkwein, Stefan},
  journal={arXiv},
  pages = {},
  volume = {},
  year={2025},
  doi = {10.48550/arXiv.2507.11130}
}

@article{karcher2018certified,
  title={Certified reduced basis methods for parametrized elliptic optimal control problems with distributed controls},
  author={K{\"a}rcher, Mark and Tokoutsi, Zoi and Grepl, Martin A and Veroy, Karen},
  journal={Journal of Scientific Computing},
  volume={75},
  number={1},
  pages={276--307},
  year={2018},
  doi={10.1007/s10915-017-0539-z},
  publisher={Springer}
}

@article{karcher2018reduced,
  title={Reduced basis approximation and a posteriori error bounds for 4D-Var data assimilation},
  author={K{\"a}rcher, Mark and Boyaval, S{\'e}bastien and Grepl, Martin A and Veroy, Karen},
  journal={Optimization and Engineering},
  volume={19},
  number={3},
  pages={663--695},
  year={2018},
  doi = {10.1007/s11081-018-9389-2},
  publisher={Springer}
}

@article{bader2016certified,
  title={Certified reduced basis methods for parametrized distributed elliptic optimal control problems with control constraints},
  author={Bader, Eduard and Kärcher, Mark and Grepl, Martin A and Veroy, Karen},
  journal={SIAM Journal on Scientific Computing},
  volume={38},
  number={6},
  doi = {10.1137/16M1059898},
  pages={A3921--A3946},
  year={2016},
  publisher={SIAM}
}

@article{himpe2018hierarchical,
  title={Hierarchical approximate proper orthogonal decomposition},
  author={Himpe, Christian and Leibner, Tobias and Rave, Stephan},
  journal={SIAM Journal on Scientific Computing},
  volume={40},
  number={5},
  pages={A3267--A3292},
  year={2018},
  doi = {10.1137/16M1085413},
  publisher={SIAM}
}

@article{troltzsch2009pod,
  title={{POD} a-posteriori error estimates for linear-quadratic optimal control problems},
  author={Tr{\"o}ltzsch, Fredi and Volkwein, Stefan},
  journal={Computational Optimization and Applications},
  volume={44},
  number={1},
  pages={83--115},
  year={2009},
  doi = {10.1007/s10589-008-9224-3},
  publisher={Springer}
}

@article{hinze2005variational,
  title={A variational discretization concept in control constrained optimization: the linear-quadratic case},
  author={Hinze, Michael},
  journal={Computational Optimization and Applications},
  volume={30},
  number={1},
  pages={45--61},
  year={2005},
  doi = {10.1007/s10589-005-4559-5},
  publisher={Springer}
}

@book{Eva10,
    author = {Evans, L.C.},
    title = {Partial Differential Equations},
    series = {Graduate Studies in Mathematics},
    volume = {19},
    publisher = {American Mathematical Society},
    address = {Providence, Rhode Island},
    year = {2010},
    doi = {},
    url = {https://bookstore.ams.org/gsm-19-r}}

@article{ran,
  title={An optimal control approach to a posteriori error estimation in finite element methods},
  author={Becker, Roland and Rannacher, Rolf},
  journal={Acta numerica},
  volume={10},
  pages={1--102},
  year={2001},
  publisher={Cambridge University Press},
  doi = {10.1017/S0962492901000010}
}

@incollection{GubV17,
    author = {Gubisch, M. and Volkwein, S.},
    title = {Proper Orthogonal Decomposition for Linear-Quadratic Optimal Control},
    booktitle = {Model Reduction and Approximation: Theory and Algorithms},
    publisher = {Society for Industrial and Applied Mathematics},
    year = 2017,
    pages = {3--63},
    doi = {https://doi.org/10.1137/1.9781611974829.ch1},
}

@article{kartmann2024certifiedmodelpredictivecontrol,
      title={Certified Model Predictive Control for Switched Evolution Equations using Model Order Reduction}, 
      author={Michael Kartmann and Mattia Manucci and Benjamin Unger and Stefan Volkwein},
      volume={},
      number={},
      pages={},
      year={2024},
      journal={arXiv},
      doi={10.48550/arXiv.2412.12930}, 
}

@book{hinze2008optimization,
  title={Optimization with PDE Constraints},
  author={Hinze, Michael and Pinnau, Ren{\'e} and Ulbrich, Michael and Ulbrich, Stefan},
  volume={23},
  year={2009},
  doi = {10.1007/978-1-4020-8839-1},
  publisher={Springer},
  address = { Dordrecht},
}

@STRING{Computing = {Computing}}

@STRING{IEEE = {{IEEE}}}

@STRING{SIAM = {Society for Industrial and Applied Mathematics}}

@STRING{Springer = {Springer-Verlag}}

\appendix
\section{Proof of Lemma \ref{lem:lemApost}}\label{ap:prooflemma4}
\begin{proof}
    We have the error-residual relation
    \begin{align*}
        \calQ(u-\tilde u)=d-\calQ\tilde u.
    \end{align*}
    Testing this equation with the error $u-\tilde u\in \calU$ and using the coercivity of $\calQ$ results in 
    \begin{align*}
        \beta \norm{u-\tilde u}{\calU}^2 \leq \langle \calQ(u-\tilde u),u-\tilde u\rangle_{\calU',\calU} = \langle d-\calQ\tilde u,u-\tilde u\rangle_{\calU',\calU} \leq \norm{d-\calQ\tilde u}{\calU'}\norm{u-\tilde u}{\calU},
    \end{align*}
    which implies the upper bound in \eqref{eq:relation_coercive_est_standard}. The lower bound follows from testing the error equation with the Riesz representative of the residual $\calJ_\calU^{-1}(d-\calQ\tilde u)\in\calU$ and using the continuity of $\calQ$:
    \begin{align*}
        \norm{d-\calQ\tilde u}{\calU'}^2 = \langle \calQ(u-\tilde u),\calJ_\calU^{-1}(d-\calQ\tilde u)\rangle_{\calU',\calU} \leq \norm{\calQ}{}\norm{u-\tilde u}{\calU}\norm{d-\calQ\tilde u}{\calU'},
    \end{align*}
    which implies the lower bound in \eqref{eq:relation_coercive_est_standard}.
\end{proof}
%
\end{document}